\documentclass[a4paper,10pt]{amsart}
\usepackage[centertags]{amsmath}
\usepackage{amsfonts}
\usepackage{amssymb}
\usepackage{amsthm}
\usepackage{newlfont}
\usepackage{graphicx}
\usepackage[usenames,dvipsnames]{color}
\usepackage{epstopdf}
\usepackage{cite}
\usepackage{bm}

\newcommand{\Q}{\mathbb{Q}}

\newcommand{\Z}{\mathbb{Z}}

\newcommand{\bq }{\begin{equation}}
\newcommand{\eq }{\end{equation}}
\theoremstyle{plain}
\newtheorem{thm}{Theorem}
\newtheorem{lem}[thm]{Lemma}
\newtheorem{prop}[thm]{Proposition}

\newtheorem{cor}[thm]{Corollary}

\newtheorem{rem}[thm]{Remark}
\theoremstyle{definition}

\theoremstyle{example}

\title{plane non-singular curves with an element of ``large'' order in its automorphism group}

\author[E. Badr] {Eslam Badr}
\address{$\bullet$\,\,Eslam Badr}
\address{Departament Matem\`atiques, Edif. C, Universitat Aut\`onoma de Barcelona\\
08193 Bellaterra, Catalonia}
\email{eslam@mat.uab.cat}
\address{Department of Mathematics,
Faculty of Science, Cairo University, Giza-Egypt}
\email{eslam@sci.cu.edu.eg}

\author[F. Bars] {Francesc Bars}
\address{$\bullet$\,\,Francesc Bars}
\address{Departament Matem\`atiques, Edif. C, Universitat Aut\`onoma de Barcelona\\
08193 Bellaterra, Catalonia} \email{francesc@mat.uab.cat}
\thanks{E. Badr and F. Bars are supported by MTM2013-40680-P}



\begin{document}

\maketitle

\begin{abstract} Let $M_g$ be the moduli space of smooth, genus $g$
curves over an algebraically closed field $K$ of zero
characteristic. Denote by ${M_g(G)}$ the subset of $M_g$ of curves
$\delta$ such that $G$ (as a finite non-trivial group) is isomorphic
to a subgroup of $Aut(\delta)$ and let $\widetilde{M_g(G)}$ be the
subset of curves $\delta$ such that $G\cong Aut(\delta)$ where
$Aut(\delta)$ is the full automorphism group of $\delta$. Now, for
an integer $d\geq 4$, let $M_g^{Pl}$ be the subset of $M_g$
representing smooth, genus $g$ plane curves of degree $d$ (in this
case, $g=(d-1)(d-2)/2$) and consider the sets
$M_g^{Pl}(G):=M_g^{Pl}\cap M_g(G)$ and
$\widetilde{M_g^{Pl}(G)}:=\widetilde{M_g(G)}\cap M_g^{Pl}$.

%Denote by $\Z/m$ a cyclic group of order $m$. First, in this note,
%once we fix $d$, we determine an algorithm in order to list the $m$
%such that $M_g^{Pl}(\mathbb{Z}/m)$ is non-empty, in particular $m$
%should divide one of the following numbers: $d-1$,$d$,$d^2-3d+3$,
%$(d-1)^2$, $d(d-2)$ or $d(d-1)$.
 In this note we first determine, for an arbitrary but a fixed
degree $d$, an algorithm to list the possible values $m$ for which
$M_g^{Pl}(\mathbb{Z}/m)$ is non-empty where $\Z/m$ denotes the
cyclic group of order $m$. In particular, we prove that $m$ should
divide one of the integers: $d-1$, $d$, $d^2-3d+3$, $(d-1)^2$,
$d(d-2)$ or $d(d-1)$. Secondly, consider a curve $\delta\in
M_g^{Pl}$ with $g=(d-1)(d-2)/2$ such that $Aut(\delta)$ has an
element of ``very large'' order, in the sense that this element is
of order $d^2-3d+3$, $(d-1)^2$, $d(d-2)$ or $d(d-1)$. Then we
investigate the groups $G$ for which
$\delta\in\widetilde{M_g^{Pl}(G)}$ and also we determine the locus
$\widetilde{M_g^{Pl}(G)}$ in these situations. Moreover, we work
with the same question when $Aut(\delta)$ has an element of
``large'' order $\ell d$, $\ell (d-1)$ or $\ell(d-2)$ with $\ell\geq
2$ an integer.

\end{abstract}
%%%%%%%%%%%%%%%%%%%%%%%%%%%%%%%%%%%%%%%%%
\section{Introduction}
{It is well known that} any $\delta\in M_g^{Pl}(G)$ corresponds to a
set {$\{C_{\delta}\}$ }of non-singular plane models in
$\mathbb{P}^2(K)$ such that any two of them are {$K$-isomorphic
through a projective transformation $P\in PGL_3(K)$} (where
$PGL_3(K)$ is the classical projective linear group of $3\times 3$
invertible matrices over $K$), and their automorphism groups are
conjugate. {If $C$ is a non-singular plane model of $\delta$ which
is defined by the homogenous equation $F(X;Y;Z)=0$ then $Aut(C)$ is
a finite subgroup of $PGL_3(K)$ and also we have $\rho(G)\preceq
Aut(C)$ for some injective representation $\rho:G\hookrightarrow
PGL_3(K)$. Moreover, $\rho(G)=Aut(C)$ whenever
$\delta\in\widetilde{M_g^{Pl}(G)}$.}
\par We denote by $\rho(M_g^{Pl}(G))$ the set of all elements $\delta\in
M_g^{Pl}(G)$ such that $G$ acts on a plane model associated to
$\delta$ as ${P^{-1}\rho(G)P}$ for some $P$ {inside $PGL_3(K)$}.
This gives us the following disjoint union decomposition:
$$M_g^{Pl}(G)=\cup_{[\rho]\in A}\rho (M_g^{Pl}(G))$$
where  $A:=\{\rho\,\,|\,\,\rho:G\hookrightarrow PGL_3(K)\}/\sim$
such that $\rho_a\sim\rho_b$ if and only if
$\rho_a(G)={P^{-1}\rho_b(G)P}$ for some $P\in PGL_3(K)$. A similar
decomposition follows for $\widetilde{M_g^{Pl}(G)}$.

{For a fixed degree $d$}, it is a difficult task to list the
$[\rho]'s$ and the groups $G$ such that $\rho(M_g^{Pl}(G))$ is
non-empty, see Henn work \cite{He} and Komiya-Kuribayashi work
\cite{KuKo2} for degree 4 and \cite{BaBa3} for degree 5. {For a
cyclic group $G\cong\Z/m$ of order $m$}, Dolgachev in \cite{Dol}
determined the $[\rho]'s$ and $m$ such that
${\rho(M_3^{Pl}(\Z/m\Z))}$ is non-trivial and {moreover he
associated} to such locus (once $\rho$ and $m$ are fixed), a normal
form, i.e. a certain projective equation which depends on some
parameters together with some algebraic restrictions to these
parameters such that any element of the locus
$\rho(M_g^{P^l}(\Z/m\Z))$ corresponds to certain specialization of
the parameters. In \S2, following Dolgachev {technique}, we obtain a
general algorithm in order to determine $[\rho]'s$ and $m$ such that
${\rho(M_g^{Pl}(\Z/m\Z))}$ {might} be non-trivial and also to such
locus (once $\rho$ and $m$ are fixed) we associate a normal form
(see Remark \ref{rem21} for a link to an implementation of the
algorithm in SAGE, and the appendix for {listing} the results {that}
are given by the algorithm until degree 9). As a consequence of the
algorithm {(Theorem \ref{thm20})} we obtain that {$m$ always}
divides one of the following integers: $d^2-3d+3$, $(d-1)^2$,
$d(d-2)$ or $d(d-1)$, which {we believe that} is well-known to the
specialists.

Secondly, there is a lot of interest on non-singular curves having a
large automorphism group{:} For $K=\mathbb{C}$ a curve $\delta\in
M_g$ has large automorphism group if it has a neighborhood ({with
respect to} the complex topology) in $M_g$ such that any other curve
{inside the neighborhood} has a smaller automorphism group. {For
such situations} $\delta$ admits a model defined over $\Q$,
$\delta/Aut(\delta)$ corresponds to the projective line and the
Galois cover $\delta\rightarrow \delta/Aut(\delta)$ is a Belyi
morphism, in particular {it is ramified exactly} at 3 points (the
last property of a Belyi morphism {that is} ramified at three points
{and} is a Galois cover{,} characterizes curves with large
automorphism group). {For more details, we refer to} Wolfart
\cite{Wol}. Another notion in the literature {for} $\delta$ {to be
of} large automorphism group is assuming that
$|Aut(\delta)|>4(g-1)$. In particular, for $\delta\in M_g^{Pl}$ it
means that $|Aut(\delta)|>2(d^2-3d+2)-4$ (in this case
$\delta\rightarrow\delta/Aut(\delta)$ is a map from $\delta$ to a
projective line which is ramified at 3 or 4 points, see
\cite[p.258-260]{FK}).

The above definitions of large automorphism group are very
restrictive {to} our proposes of plane curves $\delta\in M^{Pl}_g$
and in this paper we say that an element $\sigma\in Aut(\delta)$ is
``very large'' if its order is exactly one of the integers
$d^2-3d+3$, $(d-1)^2$, $d(d-2)$ or $d(d-1)$. We say that $\sigma\in
Aut(\delta)$ is ``large'' if its order is exactly one of the
following integers: $\ell d$, $\ell (d-1)$ or $\ell(d-2)$ {for some
integer $\ell\geq 2$}.

\noindent{In what follows $\xi_m$ denotes a primitive $m$-th root of
unity in $K$ and we obtain, in particular, the following results (in
\S3.1 to \S3.4) for $\delta\in M_g^{Pl}(\Z/m)$ such that $m$ is
``very large'' in the above sense.}

\begin{thm}Let $\delta\in M_g^{Pl}$ {be} a non-singular plane curve of degree
$d\geq 4$ and {let} $\sigma\in Aut(\delta)$ {where} $\sigma$ is
``very large''. Then {one of following cases occurs.}
\begin{enumerate}
\item if $\sigma$ has order $d(d-1)$ with $d\geq 5$ then $Aut(\delta)=<\sigma>$ and $\delta$ {is}
$K$-isomorphic to $X^d+Y^d+XZ^{d-1}=0.$ In particular {for} $d\geq
5$, $M_g^{Pl}(\Z/d(d-1)\Z)$ is an irreducible locus with one
element, and
$$\widetilde{M_g^{Pl}(\Z/d(d-1))}=M_g^{Pl}(\Z/d(d-1))=\rho(M_g^{Pl}(\Z/d(d-1)))$$
where
$\rho(\Z/d(d-1)\Z)=<diag(1,\xi_{d(d-1)}^{d-1},\xi_{d(d-1)}^d)>$.
{For the case $d=4$, one can read Remark \ref{rem3.2} in \S3.1 for
further details.}
\item if $\sigma$ has order $(d-1)^2$ then $Aut(\delta)=<\sigma>$ and
$\delta$ is $K$-isomorphic to $X^d+Y^{d-1}Z+XZ^{d-1}=0.$ {Also,}
$M_g^{Pl}(\Z/(d-1)^2\Z)$ is an irreducible locus with one element,
and
$$\widetilde{M_g^{Pl}(\Z/(d-1)^2)}=M_g^{Pl}(\Z/(d-1)^2)=\rho(M_g^{Pl}(\Z/(d-1)^2))$$
with
$\rho(\Z/(d-1)^2\Z)=<diag(1,\xi_{(d-1)^2},\xi_{(d-1)^2}^{(d-1)(d-2)})>$.
\item if $\sigma$ has order $d(d-2)$ then $\delta$ is $K$-isomorphic to
$X^d+Y^{d-1}Z+YZ^{d-1}=0$ and for $d\neq 4,6$ we have
$$H_d:=Aut(\delta)=<\sigma,\tau|\tau^2=\sigma^{d(d-2)}=1,\, and\,
\tau\sigma\tau=\sigma^{-(d-1)}>.$$ {Again,} $M_g^{Pl}(\Z/d(d-2)\Z))$
is an irreducible locus with one element, and
$$\widetilde{M_g(H_d)}=M_g^{Pl}(\Z/d(d-2))=\rho(M_g^{Pl}(\Z/d(d-2)))$$
where
$\rho(\Z/d(d-2)\Z)=<diag(1,\xi_{d(d-2},\xi_{d(d-2)}^{-(d-1)})>$.
{The automorphism groups for $d=4,6$ are given explicitly in \S 3.3,
Proposition \ref{prop15}}.
\item if $\sigma$ has order $d^2-3d+3$ then $\delta$ is $K$-isomorphic to the Klein
curve $K_d:X^{d-1}Y+Y^{d-1}Z+Z^{d-1}X=0$ and for $d\geq 5$ we have
$H_{K_d}:=Aut(\delta)=<\sigma,\tau|\sigma^{d^2-3d+3}=\tau^3=1\,
and\, \sigma\tau=\tau\sigma^{-(d-1)}>$. {The locus}
$M_g^{Pl}(\Z/(d^2-3d+3)\Z))$ is irreducible with one element, and
$$\widetilde{M_g(H_{K_d})}=M_g^{Pl}(\Z/(d^2-3d+3))=\rho(M_g^{Pl}(\Z/(d^2-3d+3)))$$ where
$\rho(\Z/(d^2-3d+3)\Z)=<diag(1,\xi_{d^2-3d+3},\xi_{d^2-3d+3}^{-(d-2)})>$.
{We refer to Remark \ref{rem18} of \S 3.4 for the classical case
$d=4$.}
\end{enumerate}
\end{thm}

\begin{rem} The above situations {do not fit} with curves {that have}
large automorphism group {in the classical definition.} For example,
the curve $\delta:X^d+Y^{d-1}Z+XZ^{d-1}=0$ is defined over $\Q$,
$\delta/Aut(\delta)$ is a projective line and the morphism
$\delta\rightarrow \delta/Aut(\delta)$ is ramified at two points of
ramification index $(d-1)^2$ and at $d-1$-points of ramification
index $d-1$. Therefore this curve has no a large automorphism group
in any of the classical sense because it ramifies at more than 4
points. But it has ``very large'' elements in its automorphism
group.
\end{rem}
{Now assuming that $m$ is ``large'' in the sense that $m\in\{\ell
d,\ell (d-1),\ell(d-2):\,\,\ell\geq 2\}$, we obtain different
results in \S4 and \S5, some of them are listed below:}
\begin{thm}Let $\delta\in M_g^{Pl}$ {be} a non-singular plane curve of degree
$d\geq 4$ {that admits an automorphism $\sigma\in Aut(\delta)$ of
``large'' order.} Then
\begin{enumerate}
\item if $\sigma$ has order $\ell (d-1)$ with $\ell\geq 2$,
{we always have} $d\equiv 0\, or\, 1\,(mod\ \ell)$ and $Aut(\delta)$
is cyclic of order $\ell'(d-1)$ with $\ell|\ell'$. If $\ell=1$, the
same conclusion holds if $\sigma$ is a homology ({By a homology we
mean that} {$P^{-1}\rho(\sigma) P=diag(1,\xi_m^a,\xi_m^b)$ with
$a=0$ or $b=0$} for some $P\in PGL_3(K)$).
\item if $\sigma$ has order $\ell d$ with $\ell\geq 3$ then $d\equiv 1\,or\, 2\,(mod\,\ell)$, $Aut(\delta)$ fixes
a line and a point off that line (in particular, following the same
notations of \S3, it is an exterior group as in Theorem
\ref{teoHarui} (2) with $N$ of order $d$). When $\ell=2$,
$Aut(\delta)$ could also be conjugate to a subgroup of $Aut(F_d)$
where $F_d$ is the Fermat curve $X^d+Y^d+Z^d=0$ (in such cases we
say that $\delta$ is a descendent of the Fermat curve, see the
precise definition in \S3).
\item if $\sigma$ has order $\ell (d-2)$ with $\ell\geq 2$ then always $d\equiv 0(mod\ \ell)$ {and, roughly speaking,} for $d>6$
{and $d\neq 10,$ we can think about $Aut(\delta)$ in a short exact
sequence} $1\rightarrow\Z/k\rightarrow Aut(\delta)\rightarrow
D\rightarrow 1$ with $k$ divides $d$ and $D$ is the Dihedral group
$D_{2(d-2)}$ or $D_{d-2}$. {For more accurate details, we refer to
\S 4.2}
\end{enumerate}
\end{thm}
 \begin{rem} In the above situations {where $m$ is ``large'', we also obtain} that every element in $M_g^{Pl}(\Z/m)$
 is {given} by a certain specialization of the parameters {in a fixed normal form for} the full locus $M_g^{Pl}(\Z/m)$.
 {This phenomena is not true in general for an arbitrary $m$. In other words, with the aid of the algorithm in \S2,
 we prove that $\rho(M_g^{Pl}(\Z/m))$ has the property of being represented by an unique fixed normal form.
 But the moduli $M_g^{Pl}(\Z/m)$ with $m$ not ``large'' or ``very large'' is not
 in general given by a single equation with some parameters (counter examples are provided in \cite{BaBacyc1})}.
\end{rem}
\begin{rem}
 Take $K=\mathbb{C}$. Then, {one} should
 expect {to have non-singular plane} curves {which have} a ``large'' element in the
 automorphism group and {no plane model} (up to $\mathbb{C}$-isomorphism) defined over the
 algebraic closure of $\Q$ inside $\mathbb{C}$. {Let us
 reproduce the situation that has been mentioned in \cite[\S2.1]{BaBacyc1} for $d=5$ and a ``large'' element of order $8$,
 as an explicit example of the above phenomena.} Any element in $M_6^{Pl}(\Z/8)$ has{, up to $K$-isomorphism,} a plane models
 of the form $X^5+Y^4Z+XZ^4+\beta
X^3Z^2=0$ for certain/s $\beta$ ({note that} $\beta\neq\pm2$ {for
non-singularity}). We constructed in \cite{BaBacyc1} a bijection map
$$\varphi:M_6^{Pl}(\Z/8\Z)\rightarrow \mathbb{A}^1(K)\setminus\{-2,2\}/\sim$$
$$\alpha\mapsto [\beta]=\{\beta,-\beta\}$$
where $a\sim b\Leftrightarrow b=a\ or\ a=-b$, and we know that the
non-singular plane model $X^5+Y^4Z+XZ^4+\beta X^3Z^2=0$ has a bigger
automorphism group than $\Z/8\Z$ if and only if $\beta=0$.
\end{rem}

\section{Cyclic automorphism group of non-singular plane curves}
Fix and integer $d\geq 4$, and consider $\delta\in M_{g}^{Pl}$ such
that the group $G\cong Aut(\delta)$ is non-trivial. Let
$C:\,\,F(X;Y;Z)=0$ in $\mathbb{P}^2(K)$ be a non-singular plane
model of degree $d$ over an {algebraically} closed field $K$ of
characteristic zero. {Suppose that} $Aut(C)=\rho(G)\leq PGL_3(K)$
for some $\rho:G\hookrightarrow PGL_3(K)$ (any other {plane} model
of $\delta$ is given by ${C_P}:\,F(P(X;Y;Z))=0$ {for some $P\in
PGL_3(K)$ moreover $Aut({C_P})$ is conjugate through $P$ to
$Aut(C)$}, and we say that ${C_P}$ is $K$-equivalent or
$K$-isomorphic to $C$). Assume that $\rho(\sigma)\in Aut(C)$ is an
element of order $m$ hence by a change of variables  in
$\mathbb{P}^2$ (in particular, changing the plane model to a
$K$-equivalent one associated to $\delta$), we can consider
$\rho(\sigma)$ as the automorphism $(x:y:z)\mapsto (x:\xi_m^a
y:\xi_m^b z)$ where $\xi_m$ is a primitive $m$-th root of unity in
$K$ and $a,b$ are integers such that $0\leq a{\neq} b\leq m-1$.
Moreover, if $ab\neq 0$ then $m$ and $gcd(a,b)$ are coprime (we can
reduce to $gcd(a,b)=1$) and if $a=0$ then $gcd(b,m)=1$. Also, such
an automorphism is identified with type $m,(a,b)$ and we write
$\rho_{a,b,m}(\Z/m\Z)$ for the subgroup given by the diagonal matrix
$diag(1,\xi_m^a,\xi_m^b)$ in $PGL_3(K)$. {In} particular $\delta\in
\rho_{a,b,m}(M_g^{Pl}(\Z/m))$ and $\delta\in
\widetilde{\rho(M_g^{Pl}(G))}$, of course $\rho_{a,b,m}$ may be
interpreted as the restriction to $<\sigma>$ of $\rho$.

Our aim here is to {investigate which cyclic groups could appear
inside} $Aut(\delta)$, {thus to determine all possible types
$m,(a,b)$ for which the moduli $\rho_{a,b,m}(M_g^{Pl}(\Z/m))$ might
be non-empty}. We follow a similar approach {as Dolgachev in
\cite{Dol} which deal} with the same question for $d=4$ (see also
\cite[\S 2.1]{Bars}).

%Let $\sigma\in Aut(C)$ be of maximal order $m$, then we can assume,
%without any loss of generality, that $\sigma$ is given in its canonical Jordan form by $(x:y:z)\mapsto
%(x:\xi_m^a y:\xi_m^b z)$ where $\xi_m$ is a primitive $m$-th root of
%unity in $K$ and $a,b$ integers such that $0\leq a\neq b\leq m-1$
%with $a\leq b$ and $gcd(a,b)$ coprime with $m$ if $ab\neq 0$(
%we can reduce to $gcd(a,b)=1$) and $gcd(b,m)=1$ if $a=0$.

Throughout this paper, {we use} the following notations.
\begin{itemize}
\item Type $m, (a,b)$ is identified with the corresponding automorphism $[X;\zeta_m^aY;\zeta_m^bY]$
where $\zeta_m$ is a primitive $m$-th root of unity. Saying that $m,
(a,b)$ is a generator of {$\rho(\Z/m)$} for certain
$\rho:\Z/m\hookrightarrow PGL_3(K)$ means that any element of
{$\rho(\Z/m)$} is a power of the {associated} automorphism {to} Type
$m,(a,b)$.
\item $L_{i,*}$ denotes a degree $i$, homogeneous
polynomial in $K[X,Y,Z]$ {such that the variable $*\in\{X,Y,Z\}$
does not appear.}
\item $S(u)_m:=\{j:\,\,u\leq j\leq d-1,\,\,d-j=0\,(mod\,\,m)\}$.
\item $S_u^{d,X}\,\,{m,(a,b)}:=\{i:\,\,u\leq i\leq d-u\,\,\text{and}\,\,ai+(d-i)b=0\,(mod\,\,m)\}$.
\item $S_u^{d-1,X}\,\,{m,(a,b)}:=\{i:\,\,1\leq i\leq d-u\,\,\text{and}\,\,ai+(d-1-i)b=0\,(mod\,\,m)\}$
\item $S(1)^{j,X}_{m,(a,b)}:=\{i:\,\,0\leq i\leq j\,\,\text{and}\,\,ai+(j-i)b=a\,(mod\,\,m)\}$.
\item $S(2)^{j,X}_{m,(a,b)}:=\{i:\,\,0\leq i\leq j\,\,\text{and}\,\,ai+(j-i)b=0\,(mod\,\,m)\}$.
\item $S^{j,Y}_{m,(a,b)}:=\{i:\,\,0\leq i\leq j\,\,\text{and}\,\,bi+(d-j)a=a\,(mod\,\,m)\}$.
\item $S^{j,Z}_{m,(a,b)}:=\{i:\,\,0\leq i\leq j\,\,\text{and}\,\,ai+(d-j)b=a\,(mod\,\,m)\}.$
\item $\Gamma_m:=\{(a,b)\in \mathbb{N}^2:\,\,g.c.d\,(a,b)=1,\,\,\,1\leq a\neq b\leq m-1\}.$
\item the points {$P_1:=(1:0:0), P_2:=(0:1:0)$ and
$P_3:=(0:0:1)$ inside $\mathbb{P}^2(K)$ are called the reference points.}
\item {$\alpha\in K^*$ and it can always be $1$ by a change of variables.}
\end{itemize}
where $u,j,m,d,a$ and $b$ are {all} non-negative integers.

%In this section, we give a complete list of non-singular plane curves of an arbitrary but fixed degree $d\geq4$
%which admits a non-trivial automorphism of order $m$ such that $p\nmid m$.
%Moreover, we attach to each model an equation which is unique up to projective equivalence.
\begin{thm}\label{thm20}
Let $\delta\in M_g^{Pl}$ be a non-singular projective plane curve of
degree $d\geq 4$ over an algebraically closed field ${K}$ of zero
characteristic. If $H$ is a non-trivial cyclic subgroup of
$Aut(\delta)$ of order $m$, then $\delta\in
\rho_{a,b,m}(M_g^{Pl}(\Z/m))$ for the following list {$(1)$-$(6)$ of
values} of $a,b,m$. {We associate to each locus a normal form, that
is unique up to $K$-equivalence. Any plane model in
$\mathbb{P}^2(K)$ of an element
$\delta\in\rho_{a,b,m}(M_g^{Pl}(\Z/m))$ is obtained by a certain
specialization of the parameters in the normal form and, any
specialization of the parameters (under certain restrictions in the
parameters) gives a plane non-singular model of an element of this
locus:} \footnote{We {warn} the reader { because {it} may happen
that a projective equation, {which is obtained by a certain} type
$m(a,b)$, is not} geometrically irreducible or not non-singular for
any {specialization of the parameters}. Hence,
$\rho_{a,b,m}(M_g^{Pl}(\Z/m))$ is the empty set and then should be
discarded from the list.}
\begin{enumerate}
  \item The curve $\delta\in\rho_{m,0,1}(M_g^{Pl}(\Z/m))$ with
  $m|d-1$ and a plane model of the curve
  is of the form $$Z^{d-1}L_{1,Z}+\big(\sum_{j\in S(2)_m}Z^{d-j}L_{j,Z}\big)+L_{d,Z}.$$

  \item The curve $\delta\in \rho_{m,0,1}(M_g^{Pl}(\Z/m))$ with $m|d$ and a plane model of the curve has
  the form $$Z^d+\big(\sum_{j\in S(1)_m}Z^{d-j}L_{j,Z}\big)+L_{d,Z}.$$

 \item All reference points lie on {$\delta$:}
 The curve $\delta\in \rho_{m,a,b}(M_g^{Pl}(\Z/m))$ with $m\,|\,(d^2-3d+3)$ and $(a,b)\in\Gamma_m$ such that
 $a=(d-1)a+b=(d-1)b\,(mod\,\,m)${. In} particular $\delta$ has a plane
 non-singular model where all reference points lie on it, and a plane non-singular model of
 $\delta$ is given by certain specialization of
 $\alpha,\beta_{j,i},\alpha_{i,j},\gamma_{i,j}\in K$ of the equation
 \begin{eqnarray*}
X^{d-1}Y&+&Y^{d-1}Z+\alpha Z^{d-1}X+\\
&+&\sum_{j=2}^{\lfloor\frac{d}{2}\rfloor}\,\,X^{d-j}\big(\sum_{i\in
S(1)^{j,X}_{m,(a,b)}}\beta_{j,i}Y^iZ^{j-i}\big)+Y^{d-j}\big(\sum_{i\in
S^{j,Y}_{m,(a,b)}}\alpha_{j,i}Z^iX^{j-i}\big)+Z^{d-j}\big(\sum_{i\in
S^{j,Z}_{m,(a,b)}}\gamma_{j,i}X^{j-i}Y^i\big),
\end{eqnarray*}
\item Two reference points lie on {$\delta$:} One of the following subcases {occurs}.
\begin{enumerate}
  \item[(4.1)] $\delta\in\rho_{m,a,b}(M_g^{Pl}(\Z/m))$ where $m\,|\,d(d-2)$ and $(a,b)\in\Gamma_m$
  such that $(d-1)a+b\equiv0\,(mod\,m)$ and $a+(d-1)b\equiv0\,(mod\,m)\,\,.$ Moreover, a plane model $C$ of $\delta$ is given by
  a certain specialization of the parameters of the equation
\[
X^d+\big(\sum_{j=2}^{d-1}\,\,X^{d-j}\sum_{i\in
S(2)^{j,X}_{m,(a,b)}}\beta_{j,i}Y^iZ^{j-i}\big)
+\big(Y^{d-1}Z+\alpha YZ^{d-1}+\sum_{i\in
S_2^{d,X}\,\,{m,(a,b)}}\beta_{d,i}Y^iZ^{d-i}\big)=0,
\]

\item[(4.2)] $\delta\in\rho_{m,a,b}(M_g^{Pl}(\Z/m))$ where $m|(d-1)^2$ and $(a,b)\in\Gamma_m$ such that $(d-1)a+b\equiv0\, (mod\,m)$ and $(d-1)b\equiv0\, (mod\,m)$.
Furthermore, a plane non-singular model $C$ of $\delta$ is obtained
by a certain specialization {of} the parameters of {the equation}
\begin{eqnarray*}
X^d&+&\sum_{j=2}^{d-2}\,\,X^{d-j}\big(\sum_{i\in S(2)^{j,X}_{m,(a,b)}}\beta_{j,i}Y^iZ^{j-i}\big)+X\big(\alpha Z^{d-1}+
\sum_{i\in S_1^{d-1,X}\,\,{m,(a,b)}}\beta_{(d-1),i}Y^iZ^{d-1-i}\big)+\\
&+&\big(Y^{d-1}Z+\sum_{i\in
S_2^{d,X}\,\,{m,(a,b)}}\beta_{d,i}Y^iZ^{d-i}\big)=0
\end{eqnarray*}
\item[(4.3)] $\delta\in\rho_{m,a,b}(M_g^{Pl}(\Z/m))$ where $m|(d-1)$ and $(a,b)\in\Gamma_m$ such that
$(d-1)b\equiv0\,(mod\,m)$ and $(d-1)a\equiv0\,(mod\,m)$. In such
case a plane non-singular model $C$ of $\delta$ has the form
\begin{eqnarray*}
X^d&+&\sum_{j=2}^{d-2}\,\,\big(X^{d-j}\sum_{i\in S(2)^{j,X}_{m,(a,b)}}\beta_{j,i}Y^iZ^{j-i}\big)+
\sum_{i\in S_2^{d,X}\,\,{m,\,(a,b)}}\beta_{d,i}Y^iZ^{d-i}+\\
&+&X\big(Z^{d-1}+\alpha Y^{d-1}+\sum_{i\in
S_2^{d-1,X}\,\,{m,\,(a,b)}}\beta_{(d-1),i}Y^iZ^{d-1-i}\big),
\end{eqnarray*}
\end{enumerate}

\item One reference points lie {on $\delta$:} Then $\delta\in \rho_{m,a,b}(M_g^{Pl}(\Z/m))$ with $m|\,d(d-1)$
and $(a,b)\in\Gamma_m$ such that $da\equiv0\,(mod\,m)$ and $(d-1)b\equiv0\,(mod\,m)$. Also, a plane model of $\delta$ is given by the form
\begin{eqnarray*}
X^d&+&Y^d+\sum_{j=2}^{d-2}\,\,\big(X^{d-j}\sum_{i\in S(2)^{j,X}_{m,(a,b)}}\beta_{j,i}Y^iZ^{j-i}\big)+
\sum_{i\in S_1^{d,X}\,\,{m,\,(a,b)}}\beta_{d,i}Y^iZ^{d-i}+\\
&+&X\big(\alpha Z^{d-1}+\sum_{i\in
S_1^{d-1,X}\,\,{m,\,(a,b)}}\beta_{(d-1),i}Y^iZ^{d-1-i}\big)=0
\end{eqnarray*}
\item None of the reference points lie on a plane model $C$ of $\delta$, then $\delta\in \rho_{m,a,b}(M_g^{Pl}(\Z/m))$
where $m|d$ and $(a,b)\in\Gamma_m$ such that $da\equiv0\,(mod\,m)$
and $db\equiv0\,(mod\,m)$. Furthermore, we have
   \[
X^d+Y^d+Z^d+\sum_{j=2}^{d-1}\,\,\big(X^{d-j}\sum_{i\in
S(2)^{j,X}_{m,(a,b)}}\beta_{j,i}Y^iZ^{j-i}\big) +\sum_{i\in
S_1^{d,X}\,\,{m,\,(a,b)}}\beta_{d,i}Y^iZ^{d-i}=0.
\]
     \end{enumerate}
Here, $\alpha,\,\beta_{i,j}, \gamma_{i,j}, \alpha_{i,j}$ are
parameters which specialize, for a concrete $\delta$ as above, at
values in ${K}$ with always $\alpha\neq0$.

\end{thm}
\begin{rem}\label{rem21}
The above result and its proof give an algorithm to list{,} for
every {fixed} degree $d$, all {cyclic groups} that could appear with
an equation (up to $K$-isomorphism){. For} the complete algorithm
and its implementation in SAGE, see the link
{http://mat.uab.cat/$\sim$eslam/CAGPC.sagews}{. Also} see the
appendix for a list of Types that could appear for degree $d\leq 9$
(i.e. the possible non-trivial $\rho_{m,a,b}(M_g^{Pl}(\Z/m))$ loci
for a {fixed} degree $d\leq 9$) with their equations {that are}
given by parameters. {These equations assign to specializations of
the parameters, plane models of the elements of the loci
$\rho_{m,a,b}(M_g^{Pl}(\Z/m))$.}
\end{rem}
%\begin{rem} The above result is also true when $K$ is an algebraic
%closed field of characteristic $p>0$ and we assume from the
%beginning that $m$ is always coprime with $p$, covering the cyclic
%groups of order coprime with the characteristic.
%\end{rem}

\begin{proof}
Without loss of generality{, we consider a plane model $C:
F(X;Y;Z)=0$} of $\delta$ such that the cyclic element order $m$ acts
as the diagonal matrix $diag(1,\xi_m^a,\xi_m^b)$ in the plane
equation $F(X;Y;Z)=0$. Let $\varphi$ be a generator of order
$m:=|H|$. One can choose coordinates so that $\varphi$ is
represented by
{$\big(x;y,z\big)\mapsto\big(x;\xi_m^ay,\xi_m^bz\big)$} {where $a,
b$} are integers with $0\leq a\neq b\leq m-1$\,( one can assume that
% \footnote{F: was written ''$b$ and $m$ are
%coprime''
%which was a contradiction on the tables, wrong. I replaced by the conditions that you
%may assume}
\, {$a<b$} with $gcd(b,m)=1$ if $a=0$ and with $gcd(a,b)=1$ {otherwise}):\\
\vspace{-.4cm}
\par\textbf{Case I\,:} Suppose first that $a=0$ {and} write:
$F(X;Y;Z)=\lambda Z^d+\big(\sum_{j=1}^{d-1}Z^{d-j}L_{j,Z}\big)+L_{d,Z}.$
\par If $\lambda=0$, then by non-singularity $L_{1,Z}\neq0$ and $(d-1)b=0\,(mod\,m)$. Hence, $m|d-1$ and we can take a generator $(a,b)=(0,1)$.
Therefore, by checking each monomial's invariance, {we obtain that
$L_{j,Z}\neq0$ only if $j\in S(2)_m$ and we get types $m,\,(0,1)$ of
$(1)$.}
\par If $\lambda\neq0$ then $db\equiv0(mod\, m)$. From which we {obtain} $m|d$ and $(a,b)=(0,1)$ is a generator for each such $m$.
{By the same discussion as before, we have} types $m,\,(0,1)$ of the form $Z^d+\big(\sum_{j\in S(1)_m}Z^{d-j}L_{j,Z}\big)+L_{d,Z}$, which proves $(2)$.\\
\vspace{-.4cm}
\par\textbf{Case II\,:} Suppose that $a\neq0$ then necessarily, $m>2$ {and we distinguish between} the following four subcases:
\begin{description}
\item[i.] All reference points lie in $C,$
\item[ii.] Two reference points lie in $C,$
\item[iii.] One reference point lies in $C,$
\item[iv.] None of the reference points lie in $C.$
\end{description}
\begin{itemize}
\item If all reference points lie on $C$, then the possibilities for the defining equation are
    now:
 \[
C:\,\,\sum_{j=1}^{\lfloor\frac{d}{2}\rfloor}\,\,\big(X^{d-j}L_{j,X}+Y^{d-j}L_{j,Y}+Z^{d-j}L_{j,Z}\big).
\]
Because $a\neq b$ with $a\neq 0$, we can assume that $C:
X^{d-1}Y+Y^{d-1}Z+\alpha
Z^{d-1}X+\sum_{j=2}^{\lfloor\frac{d}{2}\rfloor}\,\,\big(X^{d-j}L_{j,X}+Y^{d-j}L_{j,Y}+Z^{d-j}L_{j,Z}\big).$
The first three factors implies that
$a\equiv(d-1)a+b\equiv(d-1)b\,(mod\,m).$ In particular,
{$m|d^2-3d+3$. The defining equation $(3)$ follows immediately by
checking monomials' invariance in each $L_{j,B}$. For example,
rewrite $L_{j,X}$ as $\sum_{i=0}^j \beta_{j,i}Y^iZ^{j-i}$ then
$\beta_{j,i}=0$ if $m\nmid ai+(j-i)b$ or equivalently $i\notin
S(1)_{m,(a,b)}^{j,X}$, since $diag(1;\xi_m^a;\xi_m^b)\in
Aut(C)$.}\footnote {It is to be noted that for a fixed $m$ and
$(a_0,b_0)\in L_m$ where
$L_m:=\{(a,b)\in\Gamma_m:\,\,a\equiv(d-1)a+b\equiv(d-1)b\,(mod\,m)\}$,
the type $m, (a_0,b_0)$ is $K$-isomorphic to any type $m, (a',b')\in
<m, (a,b)>$. So, to complete the classification for $m$, it suffices
to choose another $(a,b)\in L_m-<(a_0,b_0)>$ and repeat until we get
$L_m=\emptyset$.}

\item If two reference points lie {on} $C$, then by re-scaling the matrix $\varphi$
and permuting the coordinates, we can assume that $(1;0;0)\notin C.$
The equation is then $C:
X^d+X^{d-2}L_{2,X}+X^{d-3}L_{3,X}+...+XL_{d-1,X}+L_{d,X}=0,$ since
$L_{1,X}$ is not invariant by $\varphi$ {because} $ab\neq0$.
Moreover, $Z^d$ and $Y^d$ are not in $L_{d,X},$ by the assumption
that only $(1;0;0)\notin C.$ Assume first that $Y^{d-1}Z$ and
$YZ^{d-1}$ are in $L_{d,X}.$ Then $(d-1)a+b{\equiv}0\,(mod\,m)$ and
$a+(d-1)b{\equiv}0\,(mod\,m)$. In particular, {$m|d(d-2)$ and for
each such type $m, (a,b)$, the equation is reduced to $
X^d+\big(\sum_{j=2}^{d-1}\,\,X^{d-j}\sum_{i=0}^{j}\beta_{j,i}Y^iZ^{j-i}\big)
+\big(Y^{d-1}Z+\alpha
YZ^{d-1}+\sum_{i=2}^{d-2}\beta_{d,i}Y^iZ^{d-i}\big)=0. $ It is
straightforward to see that if $i\notin S(2)^{j,X}_{m,(a,b)}$ (resp.
$i\notin S_2^{d,X}\,\,{m,(a,b)}$) then $\beta_{j,i}=0$ (resp.
$\beta_{di}=0$). This proves $(4.1)$.} {Secondly,} assume that
$Y^{d-1}Z\in L_{d,X}$ and $YZ^{d-1}\notin L_{d,X}.$ Then, by the
non-singularity, $Z^{d-1}$ is in $L_{d-1,X}$. That is
$(d-1)a+b\equiv0\, (mod\,m)$ and $(d-1)b\equiv0\, (mod\,m).$
Therefore {$m|(d-1)^2$ and we have the form
\[
X^d+\alpha
XZ^{d-1}+Y^{d-1}Z+\sum_{j=2}^{d-2}\,\,\sum_{i=0}^{j}\beta_{j,i}X^{d-j}Y^iZ^{j-i}+\sum_{i=1}^{d-1}\beta_{(d-1),i}XY^iZ^{d-1-i}+
\sum_{i=2}^{d-2}\beta_{d,i}Y^iZ^{d-i}=0.
\]
%\begin{eqnarray*}
%X^d&+&\sum_{j=2}^{d-2}\,\,X^{d-j}\big(\sum_{i=0}^{j}\beta_{ji}Y^iZ^{j-i}\big)+X\big(\alpha Z^{d-1}+\sum_{i=1}^{d-1}\beta_{(d-1)i}Y^iZ^{d-1-i}\big)+\\
%&+&\big(Y^{d-1}Z+\sum_{i=2}^{d-2}\beta_{di}Y^iZ^{d-i}\big)=0.
%\end{eqnarray*}
Consequently, by checking the monomials' invariance, we conclude
that if $i\notin S(2)^{j,X}_{m,(a,b)}$ then $\beta_{j,i}=0$, if
$i\notin S_1^{d-1,X}\,\,{m,(a,b)}$ then $\beta_{(d-1),i}=0$, if
$i\notin S_2^{d,X}\,\,{m,(a,b)}$ then $\beta_{d,i}=0$ and the result
follows for $(4.2)$.} Up to a permutation of $Y$ and $Z$, {it
remains to consider the case for which} $Y^{d-1}Z$ and $YZ^{d-1}$
are not in $L_{d,X}.$ By the non-singularity, $Z^{d-1}$ and
$Y^{d-1}$ should be in $L_{d-1,X}$ consequently,
$(d-1)b\equiv0\,(mod\,m)$ and $(d-1)a\equiv0\,(mod\,m).$ {Therefore,
$m|(d-1)$ and the form is reduced to
\[
X^d+XZ^{d-1}+\alpha
XY^{d-1}+\sum_{j=2}^{d-2}\,\,\sum_{i=0}^{j}\beta_{j,i}X^{d-j}Y^iZ^{j-i}+\sum_{i=2}^{d-2}\beta_{d,i}Y^iZ^{d-i}+
\sum_{i=1}^{d-2}\beta_{(d-1),i}XY^iZ^{d-1-i}=0,
\]
and the equation $(4.3)$ is now clear by the fact that
$\beta_{j,i}=0$ whenever $m\nmid ai+(j-i)b$.}
\item If one reference points lie in the $C$, then by normalizing the matrix $\varphi$ and
permuting the coordinates, we can assume that $(1;0;0),\,(0;1;0)\notin C.$ We then write
\[C: X^d+Y^d+X^{d-2}L_{2,X}+X^{d-3}L_{3,X}+...+XL_{d-1,X}+L_{d,X}=0,\]
such that $Z^d\notin L_{d,X}.$ Also, by the non-singularity, we have
$Z^{d-1}\in L_{d-1,X}.$ {In particular,} $da\equiv0\,(mod\,m)$ and
$(d-1)b\equiv0\,(mod\,m)$ and $m|\,d(d-1)$. {The above equation
become
\[
X^d+Y^d+\alpha
XZ^{d-1}+\sum_{j=2}^{d-2}\sum_{i=0}^{j}\beta_{j,i}X^{d-j}Y^iZ^{j-i}+\sum_{i=1}^{d-1}\beta_{d,i}Y^iZ^{d-i}+\sum_{i=1}^{d-1}\beta_{(d-1),i}XY^iZ^{d-1-i}=0
\]
Following the same line of argument as before, we conclude $(5)$.}

\item If none of the reference points lie in $C$ then $C: X^d+Y^d+Z^d+\big(\sum_{j=2}^{d-1}X^{d-j}L_{j,X}\big)+L_{d,X}=0,$
where $L_{1,X}$ does not appear since $ab\neq0$ {and $L_{1,X}$ is
not invariant under $\varphi$.} Clearly $da\equiv db\equiv
0\,(mod\,m)$ {and therefore} $m|d$. Moreover
\[
C:\,\,X^d+Y^d+Z^d+\sum_{j=2}^{d-1}\sum_{i\in
S(2)^{j,X}_{m,(a,b)}}\beta_{j,i}X^{d-j}Y^iZ^{j-i} +\sum_{i\in
S_1^{d,X}\,\,{m,\,(a,b)}}\beta_{d,i}Y^iZ^{d-i}=0.
\]
This completes the proof of our result.
\end{itemize}
\vspace{-.7cm}
\end{proof}

\begin{cor}\label{cor5}
Let $H$ be a non-trivial cyclic subgroup of $Aut(\delta)$ where
$\delta\in M_g^{Pl}$ with $d\geq 4$. Then {the order of $H$ divides
one of the integers}
$d-1,\,\,d,\,\,d^2-3d+3,\,\,(d-1)^2,\,\,d(d-2),\,\,d(d-1)$.
{Consequently automorphisms of $\delta$ have orders} $\leq d(d-1)$.
\end{cor}

%%%%%%%%%%%%%%%%%%%%%%%%%%%%%%%%%%%%%%%%%%%%%%%%%%%%%%%%%%%%%%%%%%%%%%%%%
%%%%%%%%%%%%%%%%%%%%%%%%%%%%%%%%%%%%%%%%%%%%%%%%%%%%%%%%%%%%%%%%%%%%%%%%%
%%%%%%%%%%%%%%%%%%%%%%%%%%%%%%%%%%%%%%%%%%%%%%%%%%%%%%%%%%%%%%%%%%%%%%%%

\section{Characterization of curves {$\delta\in M_g^{Pl}$ whose} $Aut(\delta)$ has {``very large'' elements}}

We study {here} non-singular plane curves $\delta\in M_g^{Pl}$ that
admits a $\sigma\in Aut(\delta)$ of ``very large'' or ``large''
order: $d^2-3d+3,\,\,(d-1)^2,\,\,d(d-2),$ $\,d(d-1)$, $\ell(d-1)$ or
$\ell d$ with $\ell\geq 2$. In particular we are interested in
{investigating} the full automorphism group {and the corresponding
non-singular plane} equations (up to $K$-isomorphism) of such
curves.

Before a detailed study of the automorphism groups for such
$\delta$'s, we recall the following general results concerning
$Aut(\delta)$ for $\delta\in M_g^{Pl}$ which will be useful
throughout this paper. In some cases we will use the notation of the
GAP library for {small finite} groups to indicate some {of them}.

\par Because linear systems $g^2_d$ are
unique (up to multiplication by $P\in PGL_3(K)$ in $\mathbb{P}^2(K)$
\cite[Lemma 11.28]{Book}), we always {consider a non-singular plane
model $C$ of $\delta$}, which is given by a projective plane
equation $F(X;Y;Z)=0$ and $Aut(C)$ is a finite subgroup of
$PGL_3(K)$ that fixes the equation $F$ and is isomorphic to
$Aut(\delta)$. Any other plane model of $\delta$ is given by
${C_P}:\,F(P(X;Y;Z))=0$ with $Aut({C_P})={P^{-1}Aut(C)P}$ for some
$P\in PGL_3(K)$ and ${C_P}$ is $K$-equivalent or $K$-isomorphic to
$C$. By an abuse of notation, we also denote a non-singular
projective plane curve of degree $d$ by $C$. Therefore, $Aut(C)$
satisfies one of the following situations (see Mitchel \cite{Mit}
for more details):
\begin{enumerate}
\item fixes a point $Q$ and a line $L$ with $Q\notin L$ in
$PGL_3(K)$,
\item fixes a triangle {(i.e. a set of three non-concurrent lines)},
\item $Aut(C)$ is conjugate to a representation inside $PGL_3(K)$ of
one of the finite primitive group namely, the Klein group
$PSL(2,7)$, the icosahedral group $A_5$, the alternating group
$A_6$, the Hessian group $Hess_{216}$ or to one of its subgroups $Hess_{72}$ or $Hess_{36}$.
\end{enumerate}

It is classically known that if a subgroup $H$ of automorphisms of a
non-singular plane curve $C$ fixes a point on $C$ then $H$ is cyclic
\cite[Lemma 11.44]{Book}, and recently Harui in \cite[\S2]{Harui}
provided the lacked result in the literature on the type of groups
that could appear for non-singular plane curves. Before introducing
the statement of Harui, we need to define the terminology of being a
descendent of a plane curve. For a non-zero monomial $cX^iY^jZ^k$
with $c\in K\setminus\{0\}$ we define its exponent as
$max\{i,j,k\}$. For a homogenous polynomial $F$, the core of $F$ is
defined to be the sum of all terms of $F$ with the greatest
exponent. Let $C_0$ be a smooth plane curve, a pair $(C,H)$ with
$H\leq Aut(C)$ is said to be a descendant of $C_0$ if $C$ is defined
by a homogenous polynomial whose core is a defining polynomial of
$C_0$ and $H$ acts on $C_0$ under a suitable change of the
coordinate system.

\begin{thm}[Harui] \label{teoHarui} If
$H\preceq\,\,Aut(C)$ where $C$ is a non-singular plane curve of
degree $d\geq4$ then $H$ satisfies one of the following.
\begin{enumerate}
  \item $H$ fixes a point on $C$ and then cyclic.
  \item $H$ fixes a point not lying on $C$ and satisfies a short exact sequence of the form
  $1\rightarrow N\rightarrow H\rightarrow G'\rightarrow 1,$
where $N$ a cyclic group of order dividing $d$ and $G'$ (which is a
subgroup of $PGL_2(K)$) is conjugate to a cyclic group $\Z/m\Z$ of
order $m$ with $m\leq d-1$, a Dihedral group $D_{2m}$ of order $2m$
where $|N|=1$ or $m|(d-2)$, the alternating groups $A_4$, $A_5$ or
the symmetry group $S_4$.
\item $H$ is conjugate to a subgroup of $Aut(F_d)$ where $F_d$ is the Fermat curve $X^d+Y^d+Z^d$.
In particular, $|H|\,|\,6d^2$ and $(C,H)$ is a descendant of $F_d$.
\item $H$ is conjugate to a subgroup of $Aut(K_d)$ where $K_d$ is the Klein curve curve $XY^{d-1}+YZ^{d-1}+ZX^{d-1}$
hence $|H|\,|\,3(d^2-3d+3)$ and $(C,H)$ is a descendant of $K_d$.
\item $H$ is conjugate to a finite primitive subgroup of $PGL_3(K)$ {that} are mentioned above.

\end{enumerate}
\end{thm}

{Now assume, as usual, that $C$ is} a non-singular plane model of
degree $d\geq 4$ with $\sigma\in Aut(C)$ of exact order $m$ that
acts on $F(X;Y;Z)=0$ {as} $(x,y,z)\mapsto (x,\xi_m^a y,\xi_m^bz)$.
{In the next sections, mainly in the proofs, we recall the abuse of
notation of refereing to $C$ as a non-singular plane curve (up to
$K$-isomorphism) instead of being a non-singular plane model of some
$\delta\in M_g^{Pl}$. }
%\begin{rem}
%The result of Harui is stated over the complex field $\mathbb{C}$ but it is to be noted that in positive characteristic $p$ if the order of $Aut(C)$ is coprime with $p$ then all the techniques of Hauri \cite{Harui} can be applied: Hurwitz
%bound \cite{Hurwitz}, Oiakawa and Arakawa inequalities \cite{Arak, Oiakawa} and so on. In particular, the arguments of
%the previous result.
%
%\end{rem}
%
%\begin{lem}
%If a group $G$ of automorphisms of a plane curve $C$ fixes a point
%on $C$, then it is cyclic.\footnote{{\bf F:} Is it true the
%converse? Need to think on it in the next situations.}
%\end{lem}

%\subsubsection{Curves with automorphismhes of orders $d(d-1),\,d(d-2),\,(d-1)^2$ or $d^2-3d+3$}
\subsection{{The locus $M_g^{Pl}(\Z/(d(d-1)))$}.}
\mbox{} \\
The following result appears in Harui \cite[\S3]{Harui}.
\begin{prop}[Harui]\label{prop111}
{For any $d\geq 5$,} $\delta\in\widetilde{M_g^{Pl}(\Z/(d(d-1)))}$ if
and only if { $\delta$} has a plane model given by
$C:X^d+Y^d+XZ^{d-1}=0$.
\end{prop}
{Moreover we prove the following:}
%
%In what follows, we prove a more general result.

\begin{prop}\label{prop11}
{For $d\geq 4$,} $\delta\in M_g^{Pl}(\Z/d(d-1))$ if and only if
{$\delta$} has a non-singular plane model {that is} $K$-equivalent
to $C:\,X^d+Y^d+\alpha XZ^{d-1}=0$ where $\alpha\neq0$ (always we
can assume $\alpha=1$ by a $K$-isomorphic model to $C$).
Consequently, $M_g^{Pl}(\Z/d(d-1)\Z)$ is an irreducible locus with
one element. Furthermore, for $d\geq 5$,
$$\widetilde{M_g^{Pl}(\Z/d(d-1))}=M_g^{Pl}(\Z/d(d-1))=\rho(M_g^{Pl}(\Z/d(d-1)))$$ where
$\rho(\Z/d(d-1)\Z)=<diag(1,\xi_{d(d-1)}^{d-1},\xi_{d(d-1)}^d)>$.
\end{prop}
\begin{rem}\label{rem3.2} Recall that for $d=4${,} the automorphism group of $\,\,X^4+Y^4+\alpha
XZ^{3}=0$ is {isomorphic to} $\Z/4\circledcirc A_4$, the element of
$Ext^1(A_4,\Z/4)$\footnote{{We use the notation $Ext^1(A,B)$ to
represent the group $G$ (up to isomorphism) for which there is an
exact sequence of groups of the form} $1\rightarrow B\rightarrow
G\rightarrow A\rightarrow 1$.} {which is} given by $\{(\delta,g)\in
\mu_{12}\times H:\delta^4=\chi(g)\}/{\pm 1},$ where $\mu_n$ is the
group of n-th roots of unity, $H$ is the group $A_4$ and {$\chi$ is
the character $\chi:H\rightarrow \mu_3$ defined by $\chi(S)=1$ and
$\chi(T)=\rho$ where $S,T$ are generators of $H$ of order 2 and 3
respectively with the representation $H=<S,T|S^2=1,T^3=1,...>$ and
$\rho$ is a 3rd-primitive root of unity}, see \cite{He} (or also
\cite{Bars}).
\end{rem}
\begin{proof}
{If $\delta$ has a non-singular plane model which is isomorphic to
$C:\,X^d+Y^d+\alpha XZ^{d-1}=0$ then $\delta\in M_g^{Pl}(\Z/d(d-1))$
because $[X;\zeta_{d(d-1)}^{d-1}Y;\zeta_{d(d-1)}^{d}Z]$ is an
element of $Aut(C)$ of order $d(d-1).$ Conversely, suppose that
$\delta\in M_g^{Pl}(\Z/d(d-1))$ and fix as usual, by an abuse of
notation, a plane non-singular model $C$ (up to $K$-isomorphism) of
$\delta$. Since $d(d-1)$ does not divide any of the integers
$d-1,\,\,d,\,\,d^2-3d+3,\,\,d(d-2),\,\,(d-1)^2$ then by Theorem
\ref{thm20}, $C$ is projectively equivalent to type
$d(d-1),\,\,(a,b)$ of the form $(5)$ for some
$(a,b)\in\Gamma_{d(d-1)}$ such that $da\equiv0\,\,mod\,d(d-1)$ and
$(d-1)b\equiv0\,\,mod\,d(d-1)$. In particular $a=(d-1)k$ and $b=dk'$
for some integers $k$ and $k'$ and since we have
$[X;\zeta_{d(d-1)}^{d-1}Y;\zeta_{d(d-1)}^{d}Z]^{d(k'-k)+k}=[X;\zeta_{d(d-1)}^{(d-1)k}Y;\zeta_{d(d-1)}^{dk'}Z]
$ then $m,\,(a,b)$ with $m=d(d-1),\,a=d-1$ and $b=d$ is a generator
of such types. Hence

}
\begin{eqnarray*}
S(2)^{j,X}_{m,(a,b)}&:=&\{i:\,\,0\leq i\leq j\,\,\text{and}\,\,(d-1)i+(j-i)d=0\,mod\,\,d(d-1)\}\\
&=&\{i:\,\,0\leq i\leq j\,\,\text{and}\,\,d(d-1)\,\,|(dj-i)\}\\
&=&\emptyset\,\,\forall j=2,...,d-2\,\,{(\text{because}\,\, 0<dj-i<d(d-1))},\\
\end{eqnarray*}
Also
%\begin{eqnarray*}
%\,X^d&+&Y^d+\sum_{j=2}^{d-2}\,\,\big(X^{d-j}\sum_{i\in S(2)^{j,X}_{m,(a,b)}}\beta_{ji}Y^iZ^{j-i}\big)+\sum_{i\in S_1^{d,X}\,\,{m,\,(a,b)}}\beta_{di}Y^iZ^{d-i}+\\
%&+&X\big(\alpha Z^{d-1}+\sum_{i\in S_1^{d-1,X}\,\,{m,\,(a,b)}}\beta_{(d-1)i}Y^iZ^{d-1-i}\big)=0
%\end{eqnarray*}
\begin{eqnarray*}
{S_1^{d,X}}&:=&\{i:\,\,1\leq i\leq d-1\,\,\text{and}\,\,(d-1)i+(d-i)d=0\,mod\,\,d(d-1)\}\\
&=&\{i:\,\,1\leq i\leq d-1\,\,\text{and}\,\,d(d-1)\,\,|d-i\}\\
&=&\emptyset\,\,{\text{because}\,\, 0<d-i<d(d-1)),}\\
{S_1^{d-1,X}}&:=&\{i:\,\,1\leq i\leq d-1\,\,\text{and}\,\,(d-1)i+(d-1-i)d=0\,mod\,\,d(d-1)\}\\
&=&\{i:\,\,1\leq i\leq d-1\,\,\text{and}\,\,d(d-1)\,\,|i\}\\
&=&\emptyset.
\end{eqnarray*}
{Therefore, by substituting in the form $(5)$ of Theorem
\ref{thm20}, $C$ is isomorphic to $X^d+Y^d+\alpha XZ^{d-1}$ where
$\alpha\neq0$. The last part is an immediate consequence of
Proposition \ref{prop111}.}
\end{proof}

\subsection{The moduli $M_g^{Pl}(\Z/(d-1)^2)$.}
\begin{prop}\label{prop12}  For any $d\geq 4$, if $\delta\in M_g^{Pl}$ has a non-singular plane model that is isomorphic to
$$C: \,\,X^d+Y^{d-1}Z+\alpha XZ^{d-1}=0$$ for some $\alpha\neq0$, then
$\delta\in\widetilde{M_g^{Pl}(\Z/(d-1)^2)}$.
\end{prop}
\begin{proof}
{The result is is well-known for $d=4$ (see \cite{He} or\cite{Bars}
for more details), so we assume that $d\geq5$. We have
$[X;\zeta_{d-1}Y;Z]\in Aut(C)$ which is a homology of order $d-1\geq
4$ hence $Aut(C)$ should fix a point, a line or a triangle ( see $\S
5$ of Mitchell \cite{Mit}). Since
$[X;\zeta_{(d-1)^2}Y;\zeta_{(d-1)^2}^{(d-1)(d-2)}Z]\in Aut(C)$ is of
order $(d-1)^2$ then also $(d-1)^2\,\,|\,\,|Aut(C)|$. Now assume
that} $Aut(C)$ fixes a triangle and neither a line nor a point is
fixed {by $Aut(C)$} then it follows by the proof of Theorem
\ref{teoHarui} (see \cite[\S 4]{Harui}), that $C$ is either a
descendent of the Fermat curve $F_d$ or the Klein curve $K_d$. But,
none of {these curves admits automorphisms} of order $(d-1)^2$,
{since elements of $Aut(F_d)$ (resp.\,\,$Aut(K_d)$) have orders at
most $2d$ (resp. $d^2-3d+3$). Secondly, if $Aut(C)$ fixes a point
not lying on $C$ then we can think about $Aut(C)$ in a short exact
sequence $1\rightarrow N\rightarrow Aut(C)\rightarrow G'\rightarrow
1$ as in Theorem \ref{teoHarui} $(2)$.} Since $|N|$ and $(d-1)^2$
are coprime, then {$(d-1)^2||G'|$ which is not possible for any of
the groups $\Z/m, A_4, S_4, A_5$ or $D_{2m}$ with $m\leq d-1$.
Consequently, $Aut(C)$ fixes a point on $C$ and hence it is cyclic
of order divisible by $(d-1)^2$ and $\leq d(d-1)$. That is, $Aut(C)$
is cyclic of order $(d-1)^2$. In particular
$\delta\in\widetilde{M_g^{Pl}(\Z/(d-1)^2)}$.} \vspace{-.3cm}
\end{proof}

{The following is an analogue of Proposition \ref{prop11}:}

\begin{prop}\label{prop13} {For $d\geq 4$,} $\delta\in M_g^{Pl}(\Z/(d-1)^2)$
if and only if $\delta$ has {a non-singular plane model which is
isomorphic to} $C:\,\,X^d+Y^{d-1}Z+\alpha XZ^{d-1}=0$ {with
$\alpha\neq0$. Therefore $M_g^{Pl}(\Z/(d-1)^2\Z)$ is an irreducible
locus with one element and
$\widetilde{M_g^{Pl}(\Z/d(d-1))}=M_g^{Pl}(\Z/d(d-1))=\rho(M_g^{Pl}(\Z/d(d-1)))$
where $\rho(\Z/d(d-1)\Z)$ is
$<diag(1,\xi_{(d-1)^2},\xi_{(d-1)^2}^{(d-1)(d-2)})>$. Furthermore,}
if $G$ is a non-cyclic automorphism group of {a non-singular plane
curve and $(d-1)^2\,|\,|G|$} then $G$ does not contain any element
of such order.
\end{prop}

\begin{proof}
{We only need to show that $\delta\in M_g^{Pl}(\Z/(d-1)^2)$ only if
$\delta$ has a non-singular plane model that is isomorphic to
$C:\,\,X^d+Y^{d-1}Z+\alpha XZ^{d-1}=0$ with $\alpha\neq0$, since the
remaining parts are immediate consequences of Proposition
\ref{prop12}. Up to projective equivalence, we consider a model $C$
of $\delta$ in $\rho(M_g^{Pl}(\Z/d(d-1)))$ and since $(d-1)^2\nmid
d-1,\,\,d,\,\,d^2-3d+3,\,\,d(d-2),\,\,d(d-1)$ then $C$ is isomorphic
to type $(d-1)^2, (a,b)$ of the form $(4.2)$ of Theorem \ref{thm20}.
In particular $(a,b)\in\Gamma_{(d-1)^2}$ such that
$(d-1)a+b\equiv0\, mod\,(d-1)^2,$ $(d-1)b\equiv0\, mod\,(d-1)^2$ and
$a=(d-1)k-k',\,b=(d-1)k'$ for some integers $k$ and $k'$. But we
have
$$[X;\zeta_{(d-1)^2}Y;\zeta_{(d-1)^2}^{(d-1)(d-2)}Z]^{(d-1)k-k'}=[X;\zeta_{(d-1)^2}^{(d-1)k-k'}Y;\zeta_{(d-1)^2}^{(d-1)k'}Z].$$
That is $m=(d-1)^2, a=1$ and $b=(d-1)(d-2)$ is a generator of such
types. Moreover }

%\begin{eqnarray*}
%X^d&+&\sum_{j=2}^{d-2}\,\,X^{d-j}\big(\sum_{i\in S(2)^{j,X}_{(d-1)^2,(a,b)}}\beta_{ji}Y^iZ^{j-i}\big)+X\big(\alpha Z^{d-1}+\sum_{i\in S_1^{d-1,X}\,\,{(d-1)^2,(a,b)}}\beta_{(d-1)i}Y^iZ^{d-1-i}\big)+\\
%&+&\big(Y^{d-1}Z+\sum_{i\in S_2^{d,X}\,\,{(d-1)^2,(a,b)}}\beta_{di}Y^iZ^{d-i}\big)=0
%\end{eqnarray*}
\begin{eqnarray*}
{S(2)^{j,X}}&:=&\{i:\,\,0\leq i\leq j\,\,\text{and}\,\,i+(j-i)(d-1)(d-2)=0\,mod\,\,(d-1)^2\}\\
&=&\{i:\,\,0\leq i\leq j\,\,\text{and}\,\,(d-1)^2\,\,|j(d-1)-di\}\\
&=&\emptyset\,\,\,\,\forall j=2,...,d-2.
\end{eqnarray*}
{The last equality follows because $(d-1)^2\,\,|j(d-1)-di$ implies
that $d-1|i$ thus $i=0$. But then we must have $(d-1)^2\,\,|j(d-1)$
which is impossible since $0<j<d-1$. Also, we have}

\begin{eqnarray*}
{S_2^{d,X}}&:=&\{i:\,\,2\leq i\leq d-2\,\,\text{and}\,\,i+(d-i)(d-1)(d-2)=0\,mod\,\,(d-1)^2\}\\
&\subseteq&\{i:\,\,2\leq i\leq d-2\,\,\text{and}\,\,d-1\,\,|i\}\\
&=&\emptyset,\\
{S_1^{d-1,X}}&:=&\{i:\,\,1\leq i\leq d-1\,\,\text{and}\,\,i+(d-1-i)(d-1)(d-2)=0\,mod\,\,(d-1)^2\}\\
&=&\{i:\,\,1\leq i\leq d-1\,\,\text{and}\,\,(d-1)^2\,\,|di\}\\
&=&\emptyset.
\end{eqnarray*}
{Substituting into equation $(4.2)$ yields that $C$ is isomorphic to
the equation $X^d+Y^{d-1}Z+\alpha XZ^{d-1}=0$ and we are done.}
\end{proof}

\subsection{The moduli $M_g^{Pl}(\Z/d(d-2))$.}

\par{Assume that $\delta\in M_g^{Pl}$ has} a non-singular plane model {isomorphic to the curve}
$C:\,\,X^d+Y^{d-1}Z+\alpha YZ^{d-1}=0$ of degree $d\geq4$. {The full
automorphism group of $\delta$ is given by the following result:}
\begin{prop}\label{prop15}
{Consider $\delta\in M_g^{Pl}$ with the above property. Therefore}
$Aut(\delta)$ is is the central extension
$<\sigma,\,\tau|\,\,\,\sigma^2=\tau^{d(d-2)}=1\,\,\,and\,\,\,\sigma\tau\sigma=\tau^{-(d-1)}>$
of $D_{2(d-2)}$ by $\mathbb{Z}/d$ {whenever $d\neq4,6$.} In
particular {$Aut(\delta)$ is of order} $2d(d-2)$. {For $d=6$,} it is
a central extension of $S_4$ by $\mathbb{Z}/6$ {thus} $|Aut(C)|=144$
{and for} $d=4$, $\delta$ is isomorphic to the Fermat quartic curve
$F_4$ {hence} $Aut(\delta)\simeq (\mathbb{Z}/4)^2\rtimes S_3$.

\end{prop}

\begin{proof}
Let $\mu\in K$ such that $\mu^{d(d-2)}=\alpha$, then $C$ is
projectively equivalent, through the transformation $[X;\mu
Y;\mu^{-(d-1)}Z]$, to the curve $C':\,X^d+Y^{d-1}Z+YZ^{d-1}=0$ {and
hence it follows, by $\S 6$ of} Harui \cite{Harui}, that $Aut(C')$
is {isomorphic to} $\mathbb{Z}_4^2\rtimes S_3$ (for $d=4$), a
central extension of $S_4$ by $\mathbb{Z}/d$ (for $d=6$) and a
central extension of $D_{2(d-2)}$ by $\mathbb{Z}/d$ ($d\neq4,6$).
{Finally, it is to be noted that $\sigma:=[X;Z;Y]$ and
$\tau:=[X;\zeta_{d(d-2)}Y;\zeta_{d(d-2)}^{-(d-1)}Z]$ generate
$Aut(C')$ for $d\neq4,6$ which completes the proof.}
\end{proof}
{Similarly to} Propositions \ref{prop11} and \ref{prop13}, {we
prove:}
\begin{prop}\label{prop17}
{A curve $\delta$ of $d\geq 4$ belongs to} $M_g^{Pl}(\Z/d(d-2))$ if
and only if it {has plane models that are} isomorphic to
$C:\,X^d+Y^{d-1}Z+YZ^{d-1}=0$. {Hence} $M_g^{Pl}(\Z/d(d-2))$ is
irreducible and {consists of} a single element. {Furthermore
$\widetilde{M_g^{Pl}(H_d)}=M_g^{Pl}(\Z/d(d-1))=\rho(M_g^{Pl}(\Z/d(d-1)))$}
where $H_d$ is the concrete central extension of $D_{2(d-2)}$ by
$\mathbb{Z}/d$ ($d\neq4,6$), a central extension of $S_4$ by
$\mathbb{Z}/d$ ($d=6$) or $\simeq (\mathbb{Z}/4)^2\rtimes S_3$
($d=4$){, which are} detailed in Proposition \ref{prop15}{, and
$\rho(\Z/d(d-2)\Z)=<diag(1,\xi_{d(d-2)},\xi_{d(d-2)}^{-(d-1)})>$.}
\end{prop}

\begin{proof}
It suffices to prove the {``only if'' statement since otherwise is
straightforward} by Proposition \ref{prop15}. {We have by Theorem
\ref{thm20} and because $d(d-2)\nmid
d-1,\,d,\,d^2-3d+3,\,(d-1)^2,\,d(d-1)$ that any plane model of
$\delta$ is isomorphic to type $d(d-2), (a,b)$ of the form $(4.1)$
of Theorem \ref{thm20}. That is $(a,b)\in\Gamma_{d(d-2)}$ such that
$(d-1)a+b\equiv0\,\,mod\,d(d-2)$ and
$a+(d-1)b\equiv0\,\,mod\,d(d-2).$ In particular, $a=k$ and $b=dk'+k$
for some integers $k$ and $k'$ such that $k$ and $dk'+k$ are coprime
and $d-2|k+k'$. Consequently we can take a generator $k=1$ and
$k'=d-3$, since
$[X;\zeta_{d(d-2)}Y;\zeta_{d(d-2)}^{d(d-3)+1}Z]^k=[X;\zeta_{d(d-2)}^kY;\zeta_{d(d-2)}^{dk'+k}Z]$.
Therefore}
%\[
%X^d+\big(\sum_{j=2}^{d-1}\,\,X^{d-j}\sum_{i\in S(2)^{j,X}_{m,(a,b)}}\beta_{ji}Y^iZ^{j-i}\big)
%+\big(Y^{d-1}Z+\alpha YZ^{d-1}+\sum_{i\in S_2^{d,X}\,\,{m,(a,b)}}\beta_{di}Y^iZ^{d-i}\big)=0,
%\]
\begin{eqnarray*}
{S(2)^{j,X}}&:=&\{i:\,\,0\leq i\leq j\,\,\text{and}\,\,i+(j-i)\big(d(d-3)+1\big)=0\,mod\,\,d(d-2)\}\\
&=&\{i:\,\,0\leq i\leq j\,\,\text{and}\,\,d(d-2)\,\,|j(d-1)-di\}\\
&=&\emptyset\,\,\,\,\forall j=2,...,d-2\,\,{(\text{because if}\,\,d(d-2)\,\,|j(d-1)-di\,\, \text{then}\,\, d|j\,\, \text{a contradiction})}\\
{S_2^{d,X}}&:=&\{i:\,\,2\leq i\leq d-2\,\,\text{and}\,\,i+(d-i)\big(d(d-3)+1\big)=0\,mod\,\,d(d-2)\}\\
&\subseteq&\{i:\,\,2\leq i\leq d-2\,\,\text{and}\,\,d-2\,\,|d-1-i\}\\
&=&\emptyset.
\end{eqnarray*}
{Hence} $C$ is isomorphic to $X^d+Y^{d-1}Z+\alpha YZ^{d-1}$ with
$\alpha\neq0$.
\end{proof}

\subsection{The moduli $M_g^{Pl}(\Z/(d^2-3d+3)$.}

The next result is well-known in the literature, see for example
\cite[\S 3]{Harui}.
\begin{prop} \label{prop16} {If} $\delta\in M_g^{Pl}$ {has a non-singular plane} model {of degree $d\geq5$ that is $K$- equivalent to}
$$C:\,\,X^{d-1}Y+Y^{d-1}Z+\alpha
Z^{d-1}X=0,$$ {where} $\alpha\neq 0$. Then $Aut(\delta)$ is
isomorphic to
{$<\tau,\,\sigma|\,\,\tau^{d^2-3d+3}=\sigma^3=1\,\,\,and\,\,\,\tau\sigma=\sigma\tau^{-(d-1)}\
>,$ a semidirect product of $\mathbb{Z}/3$ by
$\mathbb{Z}/{(d^2-3d+3)}$ and hence $Aut(\delta)$ is of order
$3(d^2-3d+3)$.}
\end{prop}

\begin{proof}
{Through the transformation $[X;\mu Y;\mu^{-(d-2)}Z]$ where $\mu$ is
defined by the equation $\alpha=\mu^{d^2-3d+3}$ in $K$, $C$ is
isomorphic to the Klein curve $K_d$}. It follows, by Harui
\cite{Harui} $\S 3$, that $Aut(C)$ is a semidirect product of
$\mathbb{Z}/3$ acting on $\mathbb{Z}/{(d^2-3d+3)}$. {Finally we note
that} $\tau:=[X;\zeta_{d^2-3d+3}Y;\zeta_{d^2-3d+3}^{-(d-2)}Z]$ and
$[Z;X;Y]$ {are generators of} $Aut(K_d)$ {and also satisfy the given
representation}.
\end{proof}
\begin{rem}\label{rem18} {The automorphism group of the Klein quartic curve is isomorphic to
$PSL_2(\mathbb{F}_7)$, the unique simple group of order 168 (see \cite{He}). This completes the result for any degree $d\geq4$.}
\end{rem}
%\footnote{F: What happens for $d=4$ is not the same?}
The following result should be well-known in the literature, we
write it for completeness.
\begin{prop}\label{prop14} We have $\delta\in
M_g^{Pl}(\Z/(d^2-3d+3))$ {only if} $\delta$ is isomorphic to the
Klein curve $$K_d:\,\,X^{d-1}Y+Y^{d-1}Z+Z^{d-1}X=0.$$ In particular,
{$M_g^{Pl}(\Z/(d^2-3d+3))$ is irreducible being a set with one
element and also
$$\widetilde{M_g^{Pl}(Aut(K_d))}=M_g^{Pl}(\Z/(d^2-3d+3))=
\rho(M_g^{Pl}(\Z/(d^2-3d+3)))
$$ where $\rho(\Z/d^2-3d+3\Z)=<diag(1,\xi_{d^2-3d+3},\xi_{d^2-3d+3}^{-(d-2)})>$.}
\end{prop}

\begin{proof}
Since $d^2-3d+3\nmid\,\,d-1,\,\,d,\,\,d(d-1),\,\,d(d-2),\,\,(d-1)^2$
for every $d\geq5$ then $C$ is {$K$-equivalent} to a plane curve of
type $d^2-3d+3,\,\,(a,b)$ {of the form $(3)$ in Theorem \ref{thm20}}
%of the form
% \begin{eqnarray*}
%X^{d-1}Y&+&Y^{d-1}Z+\alpha Z^{d-1}X+\\
%&+&\sum_{j=2}^{\lfloor\frac{d}{2}\rfloor}\,\,X^{d-j}\big(\sum_{i\in S(1)^{j,X}_{m,(a,b)}}\beta_{ji}Y^iZ^{j-i}\big)+Y^{d-j}\big(\sum_{i\in S^{j,Y}_{m,(a,b)}}\alpha_{ji}Z^iX^{j-i}\big)+Z^{d-j}\big(\sum_{i\in S^{j,Z}_{m,(a,b)}}\gamma_{ji}X^{j-i}Y^i\big),
%\end{eqnarray*}
for some $(a,b)\in\Gamma_{d^2-3d+3}$ such that
$a=(d-1)a+b=(d-1)b\,(mod\,\,d^2-3d+3)$. {In particular every
solution is of the form $a=k$ and $b=(d^2-3d+3)k'-(d-2)k$ for some
integers $k$ and $k'$. Because}
$[X;\zeta_{d^2-3d+3}Y;\zeta_{d^2-3d+3}^{d^2-4d+5}]^{k}=[X;\zeta_{d^2-3d+3}^kY;\zeta_{d^2-3d+3}^{(d^2-3d+3)k'-(d-2)k}]$,
we can take a generator $a=1$ and $b=d^2-4d+5$. Consequently
\begin{eqnarray*}
{S(1)^{j,X}}&:=&\{i:\,\,0\leq i\leq j\,\,\text{and}\,\,i+(j-i)(d^2-4d+5)=1\,mod\,\,(d^2-3d+3)\}\\
&=&\{i:\,\,0\leq i\leq j\,\,\text{and}\,\,(d^2-3d+3)\,\,|\big(j(d-2)-i(d-1)+1\big)\}\\
&=&\emptyset\,\,\,\,\forall j=2,...,\lfloor\frac{d}{2}\rfloor.
\end{eqnarray*}
{The last equality comes from the fact $|j(d-2)-i(d-1)+1|<d^2-3d+3$
then $j(d-2)-i(d-1)+1=0$. This in turns gives $d|2j-i-1$ which is
impossible because $0<2j-i-1<d$. Also}
\begin{eqnarray*}
{S^{j,Z}}&:=&\{i:\,\,0\leq i\leq j\,\,\text{and}\,\,i+(d-j)(d^2-4d+5)=1\,mod\,\,(d^2-3d+3)\}\\
&=&\{i:\,\,0\leq i\leq j\,\,\text{and}\,\,(d^2-3d+3)\,\,|\big((d-j)(d-2)-i+1\big)\}\\
&=&\emptyset\,\,\,\,{\forall
j=2,...,\lfloor\frac{d}{2}\rfloor\,\,(\text{since}\,\,0<(d-j)(d-2)-i+1<d^2-3d+3)}
\end{eqnarray*}
{Moreover}
\begin{eqnarray*}
{S^{j,Y}}&:=&\{i:\,\,0\leq i\leq j\,\,\text{and}\,\,(d^2-4d+5)i+(d-j)=1\,mod\,\,(d^2-3d+3)\}\\
&=&\{i:\,\,0\leq i\leq j\,\,\text{and}\,\,(d^2-3d+3)\,\,|\big((d-j)-(d-2)i-1\big)\}\\
&=&\emptyset\,\,\,\,{\forall j=2,...,\lfloor\frac{d}{2}\rfloor,}
\end{eqnarray*}
{since $|(d-j)-(d-2)i-1|<d^2-3d+3$ and if $(d-j)-(d-2)i-1=0$ then
$d-2|j-1$ a contradiction because always $0<j-1<d-2$.} Therefore $C$
is isomorphic to $X^{d-1}Y+Y^{d-1}Z+\alpha Z^{d-1}X$ with
$\alpha\neq0$. {The full automorphism of the Klein curve is
classified by Proposition \ref{prop16} and the second statement is
proved.}
\end{proof}

%%%%%%%%%%%%%%%%%%%%%%%%%%%%%%%%%%%%%%%%%%%%%%%%%%%%%%%%%%%%%%%%%%

\section{Characterization of {curves $\delta\in M_g^{Pl}$ whose
$Aut(\delta)$ has ``large'' elements}}

In the previous section we {proved} that if $m$ {is} ``very large'',
the moduli $M_g^{Pl}(\Z/m)$ {is given by one element, therefore are
irreducible set. In general it is difficult, for an arbitrary $m$,
to decide whether the set $M_g^{Pl}(\Z/m)$ is irreducible or not. We
introduced in \cite{BaBacyc1}} a weaker concept than irreducibility
that we {call ``}ES-irreducibility{''}, where a loci
$\rho(M_g^{Pl}(G))$ or $M_g^{Pl}(G)$ is {said to be}
$ES$-irreducible if {it is defined, up to $K$-isomorphism of plane
curves, via a single projective} equation {of} degree $d$ {together
with} certain parameters {that are associated to} the equation under
some {algebraic constraints, in other words, by an unique normal
form up to $K$-isomorphism. Also} any element of the locus
corresponds to {a specific} specialization of the parameters and
{vice versa}. In particular {the ``very large''-$m$ loci
$M_g^{Pl}(\Z/m)$ that appeared in \S2 are} ES-irreducible. It is not
true in general that $M_g^{Pl}(\Z/m)$ is ES-irreducible, see
{counter} examples in \cite{BaBacyc1}, and therefore is not
irreducible {as a subset of the moduli space $M_g$}.

We show here that a ``large''-$m$ locus $M_g^{Pl}(\Z/m)$ is
ES-irreducible and we obtain further details of such loci. {The
situations where $m\in\{\ell d,\ell (d-1)\}$ are} strongly related
to inner and outer Galois points {(we refer to \cite{Yoshihara} for
more details)} which will help in {determining, more precisely, the
automorphism groups} of these loci in some cases.

{One can read Henn \cite{He} or \cite{Bars} for the well-known
results in the literature on quartic curves. Hence, in what follows,
we assume that $d\geq 5$}.

\subsection{Outer and inner Galois points with $d\geq5$}\mbox{}\\

We are interested in non-singular plane curves $\delta\in
{M_g^{Pl}}$ of an arbitrary but {a} fixed degree $d\geq5$ whose
automorphism groups contain homologies of period $d$
(resp.\,\,$d-1$). Recall that {a} homology is {a finite planar
transformation such that by a change of variables it is the same as
certain type $m, (a,b)$ with $ab=0$ (see Mitchell \cite{Mit})}.
{When a homology $\omega$} of period $d$ or $d-1$ {is present inside
$Aut(\delta)$}, the genus of $\delta/<\omega>$ is zero and $\delta$
has {a} unique outer (resp.\,inner) Galois point $P$ (see
\cite[Lemma 3.7]{Harui} {for existence} and \cite{Yoshihara} for
{the definition of an inner or an outer Galois point as well as the
uniqueness in such cases}\footnote{An outer Galois point, if it
exists, is always unique except when the curve is isomorphic to the
Fermat curve, in such case there are exactly 3 outer Galois
points.}). {Furthermore if} a non-singular plane curve $\delta$ of
degree $d\geq 5$ has an outer (resp. {an} inner) Galois point $P${,
then $\tau(P)$ is also an outer (resp. an inner) Galois point of
$\delta$ for any $\tau\in Aut(\delta)$. Consequently if $\delta$ has
an unique inner Galois point then it should be fixed by the full
automorphism group $Aut(\delta)$ hence by \cite[Lemma 11.44]{Book},
$Aut(\delta)$ is a cyclic group provided that Char$(K)=0$.}

\subsubsection{The loci $M_g^{Pl}(\Z/\ell(d-1))$ with {$2\leq\ell\leq d$}.}

\begin{lem} {The locus $M_g^{Pl}(\Z/\ell(d-1))$ where $2\leq\ell\leq d$ is not empty only if}
$d\equiv0\,\,(mod\,\,\ell)$ or $d\equiv1\,\,(mod\,\,\ell)$.
\end{lem}

\begin{proof}
{Since $\ell(d-1)\nmid d-1,\, d,\, d^2-3d+3,\, d(d-2)$ then
$\ell(d-1)| d(d-1)$ or $(d-1)^2$ by Corollary \ref{cor5}.}
\end{proof}

\begin{prop}\label{prop30} {Assume that $d\geq5$ and $2\leq\ell\leq d$ with $d\equiv0\,\,(mod\,\,\ell)$,
then $\delta\in M_g^{Pl}(\Z/\ell(d-1))$ if and only if $\delta$ has a non-singular plane model that is $K$-isomorphic to}
\begin{equation}
C:\,\,X^d+Y^d+\alpha XZ^{d-1}+\sum_{2\leq {\ell k}\leq
d-2}\,\,\beta_{\ell k,\ell k}X^{d-{\ell k}}Y^{{\ell k}},
\end{equation}
 {In particular} $Aut(\delta)$ is a cyclic group of order divisible
by $\ell(d-1)$.
\end{prop}
\begin{proof}
{$(\Leftarrow)$ }Since $\sigma:=[X;\zeta_{\ell
(d-1)}^{d-1}Y;\zeta_{\ell (d-1)}^{\ell}Z]\in Aut(C)$ {is} of order
$\ell (d-1)$ then {$\delta\in M_g^{Pl}(\Z/\ell(d-1))$} {and
moreover} $C$ is not a descendant of the Klein curve $K_d$ because
$\ell(d-1)\nmid3(d^2-3d+3)$. {Also} $C$ is not a descendant of the
Fermat curve $F_d$, {since $2(d-1)\nmid6d^2$ and $\ell(d-1)>2d$ for
$\ell\geq3$ but $Aut(F_d)$ has elements of order at most $2d$. On
the other hand, $\sigma^{\ell}=[X;Y;\zeta_{\ell(d-1)}^{\ell^2}Z]\in
Aut(C)$ is} a homology of period $d-1\geq4$ {with center $P_3$ and
axis $Z=0$. Therefore the point $P_3$ is an inner Galois point of
$C$ (by Harui \cite[\S3]{Harui}) and it is unique (by Yoshihara
\cite[\S2, Theorem  4']{Yoshihara}) hence should be fixed by
$Aut(C)$. Consequently} $Aut(C)$ is a cyclic group of order
divisible by $\ell(d-1)$.\\
{$(\Rightarrow)$ Conversely,} $\ell (d-1)\nmid
d-1,\,d,\,d^2-3d+3,\,(d-1)^2$ or $d(d-2)$ therefore $\delta$ has a
non-singular plane model which is isomorphic to type $\ell (d-1),
(a,b)$ {of the form $(5)$ of Theorem \ref{thm20}. In particular
$(a,b)\in\Gamma_{\ell(d-1)}$ such that $\ell(d-1)|da$ and
$\ell(d-1)|(d-1)b$ therefore $a=(d-1)k$ and $b=\ell k'$ for some
integers $k$ and $k'$. If we consider any integer $m$ such that
$k\equiv m\,\,(mod\,\,\ell)$ then $[X;\zeta_{\ell
(d-1)}^{d-1}Y;\zeta_{\ell
(d-1)}^{\ell}Z]^{(k'-m)(d-1)+k'}=[X;\zeta_{\ell
(d-1)}^{k(d-1)}Y;\zeta_{\ell (d-1)}^{\ell k'}Z].$ Consequently we
can take $k=1=k'$ as a generator and we get}

%\begin{eqnarray*}
%\,\,X^d&+&Y^d+\sum_{j=2}^{d-2}\,\,\big(X^{d-j}\sum_{i\in S(2)^{j,X}_{m,(a,b)}}\beta_{ji}Y^iZ^{j-i}\big)+\sum_{i\in S_1^{d,X}\,\,{m,\,(a,b)}}\beta_{di}Y^iZ^{d-i}+\\
%&+&X\big(\alpha Z^{d-1}+\sum_{i\in
%S_1^{d-1,X}\,\,{m,\,(a,b)}}\beta_{(d-1)i}Y^iZ^{d-1-i}\big)=0
%\end{eqnarray*}
\begin{eqnarray*}
{S_1^{d,X}}&:=&\{i:\,\,1\leq i\leq d-1\,\,\text{and}\,\,(d-1) i+(d-i)\ell=0\,mod\,\,\ell (d-1)\}\\
&{=}&{\{i:\,\,1\leq i\leq d-1\,\,\text{and}\,\,\ell (d-1)|(d-1) i-(i-1)\ell\}}\\
&\subseteq&\{i:\,\,1\leq i\leq d-1\,\,\text{and}\,\,(d-1)\,\,|(i-1)\}=\{1\}.
\end{eqnarray*}
{Since $\ell(d-1)\nmid(d-1)(\ell+1)$ then $S_1^{d,X}=\emptyset.$
Also}
\begin{eqnarray*}
{S_1^{d-1,X}}&:=&\{i:\,\,1\leq i\leq d-1\,\,\text{and}\,\,(d-1) i+(d-1-i)\ell=0\,mod\,\,\ell (d-1)\}\\
&\subseteq&\{i:\,\,1\leq i\leq d-1\,\,\text{and}\,\,(d-1)\,\,|i\}=\{d-1\}.
\end{eqnarray*}
{But $\ell(d-1)\nmid(d-1)^2$ by the hypothesis on $\ell$, therefore
$S_1^{d-1,X}=\emptyset.$ Moreover}
\begin{eqnarray*}
{S(2)^{j,X}}&:=&\{i:\,\,0\leq i\leq j\,\,\text{and}\,\,(d-1)i+(j-i)\ell=0\,mod\,\,\ell(d-1)\}\\
&\subseteq&\{i:\,\,0\leq i\leq
j\,\,\text{and}\,\,(d-1)\,\,|j-i\}{=}{\{j\}\,\,(\text{since}\,\,0\leq
j-i<d-1)}
\end{eqnarray*}
{By assumption, $\sigma\in Aut(\delta)$ therefore}
$S(2)^{j,X}_{m,(a,b)}=\emptyset$ if $\ell\nmid j$ and $\{j\}$
otherwise. {Substituting into equation $(5)$ in Theorem \ref{thm20},
we obtain the defining equation $(1)$.} \vspace{-.2cm}
\end{proof}

%\begin{prop}\label{prop121} Take $\delta\in M_g^{Pl}(\Z/\ell(d-1))$
%with $d\geq5$, $\ell\geq 2$ and $d\equiv0\,\,(mod\,\,\ell)$. Then
%$\delta$ has a plane non-singular model isomorphic to
%\[
%\,\,X^d+Y^d+\alpha XZ^{d-1}+\sum_{2\leq j=\ell k\leq
%d-2}\,\,\beta_{ji}X^{d-j}Y^j
%\]
%In particular, $Aut(\delta)$ is cyclic.
%\end{prop}

{We also obtain a} similar result when $d\equiv 1\ (mod\ \ell)$:
%and we only state the results next:

\begin{prop}\label{prop31}
{Assume that $d\geq5$ and $2\leq\ell\leq d$ with
$d\equiv1\,\,(mod\,\,\ell)$, then $\delta\in M_g^{Pl}(\Z/\ell(d-1))$
if and only if $\delta$ has a non-singular plane model that is
$K$-isomorphic to}
\begin{equation}
X^d+Y^{d-1}Z+\alpha XZ^{d-1}+\sum_{2\leq {\ell k}\leq
d-2}\,\,\beta_{{\ell k},0}X^{d-\ell k}Z^{{\ell k}}
\end{equation}
{In such case,} $Aut(\delta)$ is {again} cyclic of order divisible
by $\ell(d-1)$.
\end{prop}
\begin{proof}
{$(\Leftarrow)$  We need only to redefine $\sigma$ to be the
automorphism $[X;\zeta_{\ell (d-1)}Y;\zeta_{\ell
(d-1)}^{(\ell-1)(d-1)}Z]$ and the rest of the argument will be quite similar.}\\
{$(\Rightarrow)$ It follows by Corollary \ref{cor5} that $\delta$
has a non-singular plane model which is isomorphic to type $\ell
(d-1), (a,b)$ of the form $(4.2)$ of Theorem \ref{thm20}. In
particular $(a,b)\in\Gamma_{\ell(d-1)}$ such that
$\ell(d-1)|(d-1)a+b, (d-1)b$ therefore $b=(d-1)k'$ and $a=\ell k-k'$
for some integers $k$ and $k'$. But $[X;\zeta_{\ell
(d-1)}Y;\zeta_{\ell (d-1)}^{(\ell-1)(d-1)}Z]^{\ell
k-k'}=[X;\zeta_{\ell (d-1)}^aY;\zeta_{\ell(d-1)}^{b}Z]$ therefore it
suffices to consider $k=1$ and $k'=\ell-1$ and we obtain
\begin{eqnarray*}
S_1^{d-1,X}&:=&\{i:\,\,1\leq i\leq d-1\,\,\text{and}\,\,i+(d-1-i)(\ell-1)(d-1)=0\,mod\,\,\ell (d-1)\}\\
&=&\{i:\,\,1\leq i\leq d-1\,\,\text{and}\,\,\ell(d-1)\,\,|di\}=\emptyset\,\,(\text{because}\,\,0<i<\ell(d-1)),\\
S_2^{d,X}&:=&\{i:\,\,2\leq i\leq d-2\,\,\text{and}\,\,i+(d-i)(\ell-1)(d-1)=0\,mod\,\,\ell (d-1)\}\\
&=&\{i:\,\,2\leq i\leq d-2\,\,\text{and}\,\,\ell (d-1)|di-(d-1)\}\\
&\subseteq&\{i:\,\,2\leq i\leq d-2\,\,\text{and}\,\,d-1|di\}=\emptyset\,\,(\text{because}\,\,0<i<d-1),\\
S(2)^{j,X}&:=&\{i:\,\,0\leq i\leq j\,\,\text{and}\,\,i+(j-i)(\ell-1)(d-1)=0\,mod\,\,\ell(d-1)\}\\
&=&\{i:\,\,0\leq i\leq j\,\,\text{and}\,\,\ell(d-1)|di-j(d-1)\}\\
&\subseteq&\{i:\,\,0\leq i\leq j\,\,\text{and}\,\,d-1|di\}=\{0\}.
\end{eqnarray*}
But $\ell(d-1)|j(d-1)$ whenever $i=0$ thus $\ell|j$. Therefore equation $(2)$ is obtained by substituting in the form $(4.2)$ of Theorem \ref{thm20}.
}
\vspace{-.2cm}
\end{proof}

{The following corollaries are immediate consequences of
Propositions \ref{prop30} and \ref{prop31}:}

\begin{cor}\label{cor101} The loci $M_g^{Pl}(\Z/\ell(d-1))$ with {$2\leq\ell\leq d$} and $d\geq 5$ {are}
empty or ES-irreducible given by one normal form.
\end{cor}

\begin{cor}\label{cor100}
{The automorphism group of any} $\delta{\in M_g^{Pl}(\Z/\ell(d-1))}$
with $2\leq\ell\leq {d}$ is cyclic {and always contains a homology
of period $d-1$.} In particular $\delta$ has a unique inner Galois
point.
\end{cor}

\begin{rem}\label{rem33} {The converse of Corollary \ref{cor100} is also true. In the sense that, if $C$ is a non-singular projective plane curve of degree $d\geq 5$ such that $Aut(C)$ contains a homology $\sigma$ of order $d-1$ with center $P$ then $C$ has an inner Galois point $P$ by \cite[Lemma 3.7]{Harui} and moreover it is unique by \cite[Theorem
4]{Yoshihara}. This point should be fixed by $Aut(C)$ which in turns implies that $Aut(C)$ is cyclic by \cite[Lemma 11.44]{Book}.
% On the other hand, if $\sigma$ is not a homology (i.e it is of type $m,(a,b)$
%with $ab\neq 0$) then it is not true in general that $Aut(C)$ is a cyclic group.}
.}
\end{rem}

\subsubsection{The loci $M_g^{Pl}(\Z/\ell d)$ {with $2\leq\ell\leq d-1$}.}

\begin{lem} {The locus $M_g^{Pl}(\Z/\ell d)$ where $2\leq\ell\leq d-1$ is not empty only if} $d=1\,\,(mod\,\ell)$ or
$d\equiv2\,\,(mod\,\ell)$.
\end{lem}
\begin{proof}
{The result follows by Corollary \ref{cor5}, since $\ell d\nmid
d-1,\,d,\,d^2-3d+3,\,(d-1)^2$}.
\end{proof}
\begin{prop}\label{prop32}
{Assume that $d\geq5$ and $3\leq\ell\leq d-1$ with
$d\equiv1\,(mod\,\ell)$, then $\delta\in M_g^{Pl}(\Z/\ell d)$ if and
only if $\delta$ has a non-singular plane model that is
$K$-isomorphic to}
\begin{equation}
\tilde{C}:\,X^d+Y^d+\alpha XZ^{d-1}+\sum_{2\leq \ell k\leq
d-2}\,\,\beta_{\ell k,0}X^{d-\ell k}Z^{\ell k}
\end{equation}
where $\alpha\neq0$. {In this case, $Aut(\delta)$} should fix a line
and a point off that line {and every automorphism of $\delta$ is
projectively equivalent to a transformation of the form
$[\alpha_1X+\alpha_3Z;Y;\gamma_1X+\gamma_3Z]$.}
\end{prop}

\begin{proof}
$(\Leftarrow)$ Since $\sigma:=[X;\zeta_{\ell d}^{\ell}Y;\zeta_{\ell
d}^{d}Z]\in Aut(\tilde{C})$ {is} of order $\ell d$ then {$\delta\in
M_g^{Pl}(\Z/\ell d)$ and moreover $\sigma^{\ell}\in Aut(\tilde{C})$
is} a homology of period ${d}>4$ with center $P_2$ and axis $Y=0$.
{In particular, by \cite{Mit}, $Aut(\tilde{C})$ fixes} a line and a
point off that line or it fixes a triangle. {Assume that} it fixes a
triangle and neither a point nor line is leaved invariant, then
$\tilde{C}$ {is} a descendant of the Klein curve $K_d$ or the Fermat
curve $F_d$ {which is impossible because $\ell d\nmid3(d^2-3d+3)$
and elements of $Aut(F_d)$ have orders at most $2d<\ell d$.}
   %If $\ell=2$ and $Aut(C_1)$ is conjugate to a subgroup of $Aut(F_d)$ then this occurs through an invertible element inside
%   $PGL_3(\mathbb{C})$ of the form $[\alpha_1X+\alpha_3Z;Y;\gamma_1X+\gamma_3Z]$ since
%   $P[X;\zeta_{2d}^4Y;Z]P^{-1}=[X;\zeta_{2d}^4Y;Z]$ being; $[X;\zeta_{2d}^4Y;Z]\in Aut(F_d)$(for simplicity, assume $det(P)=1$).
%   In particular, $P[X;Y;\zeta_{2d}^dZ]P^{-1}\in Aut(F_d)$ that is,
%\[\alpha_1^d\big(\gamma_3X-\alpha_3Z\big)^d+\gamma_3^d\big(\gamma_1X-\alpha_1Z\big)^d=k(X^d+Z^d).\]
%From which we must have
%\begin{eqnarray*}
%\alpha_1\alpha_3^{d-1}+\gamma_1\gamma_3^{d-1}&=&0\,\,\,(XZ^{d-1})\\
%\alpha_1^2\alpha_3^{d-2}+\gamma_1^2\gamma_3^{d-2}&=&0\,\,\,(X^2Z^{d-2})
%\end{eqnarray*}
%Consequently, $\gamma_1=\alpha_3=0$ a contradiction (because we get $-1=(\alpha_1\gamma_3)^d=1$).\\
{Consequently a line and a point off that line is leaved invariant.
Also it follows by \cite{Harui} $\S 3$,} that the point {$P_2$} is
an outer Galois point of $\tilde{C}$. Moreover {it is unique
because} $\tilde{C}$ is not isomorphic to the Fermat curve $F_d$
{(\cite{Yoshihara} $\S 2$ Theorem $4'$) hence this point should be
fixed by $Aut(\tilde{C})$. Furthermore the axis $Y=0$ should also be
fixed (see \cite{Mit}, Theorem 4) that is automorphisms of
$\tilde{C}$ are of the form $[\alpha_1X+\alpha_3Z;Y;\gamma_1X+\gamma_3Z]$.}\\
$(\Rightarrow)$ Conversely, one may follow the same line of argument
in Proposition \ref{prop30} to conclude that $\tilde{C}$ is
isomorphic to type $\ell d, (\ell k,dk')$ of the form $(5)$ of
Theorem \ref{thm20} and to figure out that we can assume $k=1=k'$ as
a generator, since $[X;\zeta_{\ell d}^{\ell}Y;\zeta_{\ell
d}^{d}Z]^{(k'-m)d+k}=[X;\zeta_{\ell d}^{\ell k}Y;\zeta_{\ell
d}^{dk'}Z]$ where $k\equiv m\,(mod\,\ell)$. In this case, we get
\begin{eqnarray*}
{S_1^{d,X}}&:=&\{i:\,\,1\leq i\leq d-1\,\,\text{and}\,\,\ell i+(d-i)d=0\,mod\,\,\ell d\}\\
&{=}&\{i:\,\,1\leq i\leq d-1\,\,\text{and}\,\,\ell d\,\,|i(d-\ell)-d\}\\
&{\subseteq}&{\{i:\,1\leq i\leq
d-1\,\,\text{and}\,\,d|i\}}=\emptyset.
\end{eqnarray*}
{Similarly $S_1^{d-1,X}\subseteq\{i:\,\,1\leq i\leq
d-1\,\,\text{and}\,\,d\,\,|i\}=\emptyset$. Furthermore $i\in
S(2)^{j,X}$ iff $\ell d|\ell i-(j-i)d$ thus $d|i$ and $i=0$. That is
$i\in S(2)^{j,X}\neq\emptyset$ only if $\ell|j$ which completes the
proof. }
\end{proof}

\begin{rem} {For $\ell=2$, proposition \ref{prop32} is true with the same proof if we assume that $\delta$ is not a descendent of the Fermat curve of
degree $d$.}
\end{rem}

{There is a similar statement to the previous results when
$\ell|d-2$. We state only the result since the proof can be obtained
through similar techniques:}
\begin{prop} {Assume that $d\geq5$ and $2\leq\ell\leq d-1$ with $d\equiv2\,(mod\,\ell)$, then $\delta\in M_g^{Pl}(\Z/\ell d)$ if and only if $\delta$ has a non-singular plane model that is $K$-isomorphic to}
\begin{equation}
\widehat{C}:\,X^d+Y^{d-1}Z+\alpha YZ^{d-1}+\sum_{2\leq i=\ell
k+1\leq d-2}\beta_{d,i}Y^iZ^{d-i}=0.
\end{equation}
{Moreover $\widehat{C}$ is a descendant of the Fermat curve $F_d$
(only if $\ell=2$) or $Aut(\delta)$ fixes a line and a point off
this line (in particular automorphisms of $\widehat{C}$ have the
form $[X;\beta_2Y+\beta_3Z;\gamma_2Y+\gamma_3Z])$.}
\end{prop}

%
%
%\begin{proof}
%Clearly, $\ell d$ is not a divisor of $d-1,\,\,d,\,\,d^2-3d+3,\,\,(d-1)^2$ or $d(d-1)$. Therefore, $C_d$ is projectively equivalent to type $\ell d, (k,dk'+k)$ of Theorem \ref{thm20} $(4.1)$ ( $k$ and $dk'+k$ are coprime and $<\ell d$ and $\ell |k+k',k+(d-1)k'$) of the form
%\[
%X^d+\big(\sum_{j=2}^{d-1}\,\,X^{d-j}\sum_{i\in S(2)^{j,X}_{m,(a,b)}}\beta_{ji}Y^iZ^{j-i}\big)
%+\big(Y^{d-1}Z+\alpha YZ^{d-1}+\sum_{i\in S_2^{d,X}\,\,{m,(a,b)}}\beta_{di}Y^iZ^{d-i}\big)=0,
%\]
%We can take a generator $k=1$ and $k'=\ell-1$ since $[X;\zeta_{\ell d}Y;\zeta_{\ell d}^{(\ell-1)d+1}Z]^{k}=[X;\zeta_{\ell d}^{k}Y;\zeta_{\ell d}^{(\ell-1)dk+k}Z]=[X;\zeta_{\ell d}^{k}Y;\zeta_{\ell d}^{-dk+k}Z]=[X;\zeta_{\ell d}^{k}Y;\zeta_{\ell d}^{dk'+k}Z]$.
%The last equality because $\ell|k+k'$. Now, we have
%\begin{eqnarray*}
%S_2^{j,X}\,\,{\ell d,(1,(\ell-1)d+1)}&:=&\{i:\,\,0\leq i\leq j\,\,\text{and}\,\,i+(j-i)\big((\ell-1)d+1\big)=0\,mod\,\,\ell d\}\\
%&\subseteq&\{i:\,\,0\leq i\leq j\,\,\text{and}\,\,\ell d\,\,|j(d-1)+(d-2)i\}\\
%&=&\emptyset
%\end{eqnarray*}
%The last equality because $\ell d\,\,|j(d-1)+di$ implies that $d|j$ a contradiction since $0<j<d.$
%\begin{eqnarray*}
%S_2^{d,X}\,\,{\ell d,(1,(\ell-1)d+1)}&:=&\{i:\,\,2\leq i\leq d-2\,\,\text{and}\,\,i+(d-i)\big((\ell-1)d+1\big)=0\,mod\,\,\ell d\}\\
%&\subseteq&\{i:\,\,2\leq i\leq d-2\,\,\text{and}\,\,\ell\,\,|i-1\}\\
%&=&\{\ell m+1:\,\,\,m=1,2,...,\frac{d-4}{\ell}\}
%\end{eqnarray*}
%\end{proof}
\begin{rem} Unfortunately {it may happen here that different families of groups
 appear as the full automorphism of $\delta\in M_g^{Pl}(\Z/\ell d)$ depending on the specialization of the parameters.}
\end{rem}
\begin{cor} {The loci $M_g^{Pl}(\Z/\ell d)$ with $2\leq\ell\leq d-1$ and $d\geq 5$ are
empty or ES-irreducible.}
\end{cor}
{It is well known by \cite[Lemma 3.7]{Harui} that if $Aut(\delta)$
has a homology of period $d$ then $\delta$ has an outer Galois
point. Moreover if $\delta$ is isomorphic to the Fermat curve of
degree $d$, then it has two more outer Galois points and it is
unique otherwise \cite[Theorem 4' and Proposition 5']{Yoshihara}.
Furthermore we conclude the following:}

\begin{cor}\label{cor44}
{For any $\delta\in M_g^{Pl}(\Z/\ell d)$ with $3\leq\ell\leq d-1$,
$Aut(\delta)$ always contains a homology of period $d$. In
particular $\delta$ has an unique outer Galois point.}
\end{cor}

%%%%%%%%%%%%%%%%%%%%%%%%%%%%%%%%%%%%%%%%%%%%%%%%%%%%%%%%%%%%%%%%%%

%%%%%%%%%%%%%%%%%%%%%%%%%%%%%%%%%%%%%%%%%%%%%%%%%%%%%%%%%%%%%%%%%%%%%%%%%%%%%%%%

%\include{l(d-2)moduliV2}
\subsection{On the loci $M_g^{Pl}(\Z/\ell(d-2)\Z)$}\mbox{}\\
{We investigate here the finite groups $G$ that contain cyclic
subgroups of order $\ell(d-2)$ and for which the locus $M_g^{Pl}(G)$
may be not empty. This question is completely solved when $d=4$ (see
\cite{He}) and $d=5$ (see \cite{BaBa3}) therefore we assume in this
part that $d\geq 6$ and also $\ell\geq2$.}
\begin{lem}\label{E1}
{The locus $M_g^{Pl}(\Z/\ell(d-2))$} with $d\geq 6$ {and $\ell\geq2$
is non-empty only if} $d\equiv0\,(mod\,\ell)$.
\end{lem}
\begin{proof}
We have $d\geq6>2+\frac{2}{\ell-1}$ therefore $\ell(d-2)>d$ and $\ell(d-2)\nmid d-1$ or $d$. Also $(d-1)^2=d(d-2)+1, d(d-1)=d(d-2)+d$ and $d^2-3d+3=(d-1)(d-2)+1$ thus $\ell(d-2)\nmid(d-1)^2, d(d-1)$ or $d^2-3d+3$, since $(d-2)\nmid d$ or $1$. Now the result follows by Corollary \ref{cor5}.
\end{proof}

%\begin{cor}
%The loci $M_g^{Pl}(\Z/\ell(d-2)\Z)$ and $\widetilde{M_g^{Pl}(\Z/\ell(d-2)\Z)}$ are empty for any odd degree $d\geq6$ where $\ell|d$ such that $2|\ell$. In particular these loci are irreducible.
%\end{cor}
%\noindent Now we study the locus $M_g^{Pl}(\Z/\ell(d-2)\Z)$ for an arbitrary but a fixed even degree $d\geq6$ with $\ell|d$ such that $2|\ell$.
{We treat first the situation when $\ell$ is even:}
\begin{prop}\label{P1} {Suppose that $\ell\geq2$ is an even integer such that $\ell|d$ with $d\geq6$.
Any $\delta\in M_g^{Pl}(\Z/\ell(d-2)\Z)$ has a plane non-singular
model of the form}
\begin{equation}\label{(Eq.1)}
X^d+Y^{d-1}Z+\alpha
YZ^{d-1}+\sum_{k=1}^{\lfloor\frac{d}{2\ell}\rfloor}\,\,\beta_{2\ell
k,\ell k}X^{d-2\ell k}Y^{\ell k}Z^{\ell k}=0
\end{equation}
{In this case, the locus $M_g^{Pl}(\Z/\ell(d-2)\Z)$ is
ES-irreducible.}
\end{prop}
\begin{proof}
If $\delta\in M_g^{Pl}(\Z/\ell(d-2)\Z)$ {then} $\delta$ has an
automorphism $\sigma$ of order $\ell(d-2)$. Consequently
${\tau}:=\sigma^{\frac{\ell}{2}}{\in Aut(\delta)}$ {is of order
$2(d-2)$, that is} $\delta\in M_g^{Pl}(\Z/2(d-2)\Z)$. Therefore we
{need only to deal with} the case $\ell=2$. It follows by Lemma
\ref{E1} that {a non-singular plane model $C_{(a,b)}$ of $\delta$
should be} isomorphic to type $2(d-2), (a,b)$ {of the form $(4.1)$
of Theorem \ref{thm20} for some $(a,b)\in \Gamma_{2(d-2)}$ and
$2(d-1)|(d-1)a+b,\,a+(d-1)b$.
%$$C_{(a,b)}:\,X^d+\big(\sum_{j=2}^{d-1}\,\,X^{d-j}\sum_{i\in S(2)^{j,X}_{m,(a,b)}}\beta_{ji}Y^iZ^{j-i}\big)
%+\big(Y^{d-1}Z+\alpha YZ^{d-1}+\sum_{i\in S_2^{d,X}\,\,{m,(a,b)}}\beta_{di}Y^iZ^{d-i}\big)=0$$
Clearly $(1,d-3)\in\Gamma_{2(d-2)}$ is a solution of this system and $[X;\xi_{2(d-2)}Y;\xi_{2(d-2)}^{d-3}Z]\in Aut(C_{(1,d-3)})$. On the other hand, $2|a-b$ and $d-2|a+b$, in particular $a=k+(\frac{d-2}{2})k'$ and $b=-k+(\frac{d-2}{2})k'$ for some integers $k$ and $k'$ and we get $2|\pm k+(\frac{d}{2})k'$. Consequently $[X;\xi_{2(d-2)}Y;\xi_{2(d-2)}^{d-3}Z]^{k+(\frac{d-2}{2})k'}=[X;\xi_{2(d-2)}^aY;\xi_{2(d-2)}^{b}Z]$ and $m=2(d-2), a=1$ and $b=d-3$ is a generator of the set of solution of our system. Furthermore the associated sets $S_2^{d,X}$ and $S(2)^{j,X}$ for $j=2,...,d-1$ are  computed as follows:}
%$$\xi_{2(d-2)}^{(d-3)(k+(\frac{d-2}{2})k')}=\xi_{2(d-2)}^{b+(d-2)k+(d-4)(\frac{d-2}{2})k'}=\xi_{2(d-2)}^{b}\xi_{2(d-2)}^{(d-2)k+d(\frac{d-2}{2})k'}
%=\xi_{2(d-2)}^{b}\xi_{2(d-2)}^{(d-2)(k+\frac{d}{2}k')}=\xi_{2(d-2)}^{b}$$
%Therefore
\begin{eqnarray*}
{S_2^{d,X}}&:=&\{i:\,\,2\leq i\leq d-2\,\,\text{and}\,\,2(d-2)|i+(d-i)(d-3)\}\\
&\subseteq&\{i:\,\,2\leq i\leq d-2\,\,\text{and}\,\,(d-2)|2(i-1)\}\\
&=&\{\frac{d}{2}\}
\end{eqnarray*}
since {$0<2(i-1)<2(d-2)$} therefore $2(i-1)=d-2.$ {Also} we have
\begin{eqnarray*}
S(2)^{j,X}_{m,(a,b)}&:=&\{i:\,\,0\leq i\leq j\,\,\text{and}\,\,2(d-2)|i+(j-i)(d-3)\}\\
&\subseteq&\{i:\,\,0\leq i\leq j\,\,\text{and}\,\,(d-2)|j-2i\},
\end{eqnarray*}
{But $|j-2i|\leq d-1$ therefore $j-2i=0$ or $\pm(d-2)$}. In
particular, $S(2)^{j,X}=\emptyset$ if $j$ is odd and
$\{\frac{j}{2},\frac{j\pm(d-2)}{2}\}$ if $j$ is even. {Moreover}
$0\leq i\leq j$ {thus when $j$ is even and $<d-2$},
${S(2)^{j,X}}=\{\frac{j}{2}\}$ and when $j=d-2$,
${S(2)^{d-2,X}=}\{0,\frac{d-2}{2},d-2\}$. Consequently, we obtain
the form
$$
X^d+Y^{d-1}Z+\alpha YZ^{d-1}+X^{2}\big(\beta_{d-2,0}Z^{d-2}+\beta_{0,d-2}Y^{d-2}\big)+\sum_{j=2,4,...,d-2,d}\,\,\beta_{j,\frac{j}{2}}X^{d-j}Y^{\frac{j}{2}}Z^{\frac{j}{2}}=0
$$
Because $[X;\xi_{\ell(d-2)}Y;\xi_{\ell(d-2)}^{d-3}Z]\in
Aut(C_{1,d-3})$ {hence} $\beta_{d-2,0}=\beta_{0,d-2}=0$ moreover
$\beta_{j,(\frac{j}{2})}=0$ if $2\ell\nmid j$. To deal $\ell>2$ even
one obtain the result y impose that the automorphism associated to
Type $\ell, (a,b)$ leaves invariant the equation.
\end{proof}
%\noindent Now, we determine the locus $\widetilde{M_g^{Pl}(G)}$ such that $\Z/\ell(d-2)\Z\preceq G$ or equivalently,
%investigating the full automorphism groups of curves which are defined by (\ref{(Eq.1)}).
\begin{prop}\label{P3}
{Let $\ell\geq2$ be an even integer such that $\ell|d$ with $d\geq6$
and let $G$ be a finite group inside $PGL_3(K)$. Then $\delta\in
M_g^{Pl}(\Z/\ell(d-2)\Z)\cap \widetilde{M_g^{Pl}(G)}$} only if one
of the following situations occurs:
\begin{enumerate}
  \item $d=6$ and $G$ is conjugate to a central extension of $S_4$ by $\Z/6\Z$.
  {In this case,} $G$ is of order $144$ and $\widetilde{M_g^{Pl}(G)}$ is {an} irreducible set {that is} given by one
  element which has a plane non-singular model of the form
$X^6+Y^{5}Z+YZ^{5}=0$.
  \item $d>6$ and $G$ is conjugate to $<\sigma,\tau|\tau^2=\sigma^{d(d-2)}=1,\,\tau\sigma\tau=\sigma^{-(d-1)}>$,
  a central extension of order $2d(d-2)$ of $D_{2(d-2)}$ by $\Z/d\Z$. {Also} $\widetilde{M_g^{Pl}(G)}$
  is an irreducible set {and is} given by one element {with a non-singular plane} model isomorphic to $X^d+Y^{d-1}Z+YZ^{d-1}=0$.
\item $d=6$ and $G$ is isomorphic to $SmallGroup(16,8)$ {in GAP library}. Furthermore any element of {$\widetilde{M_{10}^{Pl}(SmallGroup(16,8))}$
has a non-singular plane model, $K$-isomorphic, to}
$X^{6}+Y^{5}Z+YZ^{5}+{\beta_{4,2}}X^{2} Y^{2} Z^{2}=0$ for certain
${\beta_{4,2}}\neq 0$.
\item $d=10$ and $G$ is isomorphic to $SmallGroup(32,19)$ {in GAP library. Similarly $\widetilde{M_{36}^{Pl}(SmallGroup(32,19))}$
consists of a curves which has a non-singular plane model (up to
$K$-equivalence) of the form}
$X^{10}+Y^{9}Z+YZ^{9}+{\beta_{6,4}}X^{6} Y^{2}
Z^{2}+{\beta_{2,8}}X^{2} Y^{4} Z^{4}=0$ with
$(\beta_{6,4},\beta_{2,8})\neq (0,0)$.
\item $d\neq6, 10$ and $G$ is an element ${Ext^1(N,D_{2(d-2)})}$ {where $N$ is} a cyclic group of order {$2r(|d)$}. Moreover $G$
 {contains} $<\sigma,\tau:\,\tau^2=\sigma^{\ell(d-2)}=1\,\textit{and}\,\tau\sigma\tau=\sigma^{-(d-1)}>$ {as a subgroup.
 Also every element of $\widetilde{M_g^{Pl}(G)}$} has a non-singular
plane model of the form $(\ref{(Eq.1)})$ of Proposition \ref{P1}
such that $\beta_{2k\ell, \ell,k}\neq0$ for some
$k\in\{1,...,\lfloor\frac{d}{2\ell}\rfloor\}$.
\end{enumerate}
\end{prop}
\begin{proof}
It is sufficient, by Proposition \ref{P1}, to consider non-singular
plane curves that is {defined} by equation (\ref{(Eq.1)}). {First,}
assume that $\beta_{2\ell k,\ell k}=0$ for all
$k=1,...,\lfloor\frac{d}{2\ell}\rfloor$, {thus} elements of
$\widetilde{M_g^{Pl}(G)}$ have a plane model which is isomorphic to
{the form} $X^d+Y^{d-1}Z+\alpha YZ^{d-1}=0.$ The {full} automorphism
group in such case is well known {by Proposition \ref{prop15}. This
proves $(1)$ and $(2)$.}

Secondly, {suppose} that $\beta_{2\ell j,\ell j}\neq0$ for some
$j\in\{1,...,\lfloor\frac{d}{2\ell}\rfloor\}$. {It is to be noted
that the form} (\ref{(Eq.1)}) of Proposition \ref{P1} {always
admits} a a bigger automorphism group {namely,}
$G_{0}:=<\sigma,\tau>$ of order $2\ell(d-2)$ where
$\sigma:=[X;\xi_{\ell(d-2)}Y;\xi_{\ell(d-2)}^{d-3}Z]$ and
$\tau:=[X;\mu Z;\mu^{-1}Y]$ with $\mu^{d-2}=\alpha.$ Consequently
$Aut(C)$ is not cyclic{, since $G_0$ does being isomorphic to}
$<\sigma,\tau|\tau^2=\sigma^{\ell(d-2)}=1,\,
and\,\tau\sigma\tau=\sigma^{-(d-1)}>$. Also $C$ is not a descendant
of the Klein curve $K_d$ because {$|G_0|\nmid 3(d^2-3d+3)$}.
Moreover $Aut(C)$ is not conjugate to any of the finite primitive
subgroups of $PGL_3(K)$, {since} $\ell(d-2)\geq 8$ and non of these
groups contains elements of order $>7$ (in fact, the Klein group
$PSL(2,7)$ is the only primitive group in $PGL_3(K)$ with elements
of order 7). {On the other hand,} $C$ is not a descendant of the
Fermat curve, since $\ell(d-2)>2d$ for all $\ell>2$ and elements of
$Aut(F_d)$ have orders at most $2d$ also for $\ell=2$,
$4(d-2)\nmid6d^2$ because $d\geq6$ and is even.
\par Now it follows by the above argument that $Aut(C)$ should fix a line and a point off that line
{where the fixed point} does not belong to $C$. {But we have
$\sigma,\tau\in Aut(C)$ therefore the line must be $X=0$ and the
point is $P_1$. In particular,} automorphisms of $C$ are of the form
$[X;\beta'_2Y+\beta'_3Z;\gamma'_2Y+\gamma'_3Z]$ {and we can think
about $Aut(C)$ in} a short exact sequence $1\rightarrow N\rightarrow
Aut(C)\rightarrow \rho(Aut(C))\rightarrow 1$ with
$N=<diag(\xi_d^{r'};1;1)>$ a cyclic group of order dividing $d$,
$\rho(Aut(C))$ is conjugate to a cyclic group ${\Z/m\Z}$ of order
$m\leq d-1$, a Dihedral group $D_{2m}$ where $m|(d-2)$ ({recall
that} $diag(-1;1;1)\in N$), the alternating groups $A_4$, $A_5$ or
the permutation group $S_4$ and {$\rho:PBD(2,1)\hookrightarrow
PGL_2(K)$ is the canonical map where $PBD(2,1)$ is the subgroup of
$PGL_3(K)$ that all the entries in the third column and third row
are zero except the one in the diagonal which has value 1. It
suffices to consider the case $\ell=2$, since}
$M_g^{Pl}(\Z/\ell(d-2)\Z)\subseteq M_g^{Pl}(\Z/2(d-2)\Z)$. {Hence}
$\rho(Aut(C))$ contains {the} element $\rho(\tau)=\left(
  \begin{array}{cc}
0 & \mu \\
 \mu^{-1} & 0 \\
\end{array}
\right)$ of order $2$ and {the} element $\rho(\sigma)=\left(
 \begin{array}{cc}
 1 & 0 \\
 0 & \xi_{2(d-2)}^{d-4} \\
 \end{array}
  \right)$ of order {$d-2$ (only if $4\nmid d-2$) and $\frac{d-2}{2}$ (otherwise) therefore
  $\rho(Aut(C))$ always} contains a dihedral subgroup {and then it} is not conjugate to a cyclic group {$\Z/m\Z$}. {Now if $4\nmid d-2$
  (resp. $4|d-2$ and $d\neq6,10$) then $\rho(Aut(C))$ has elements of order $>5$. In particular, it is not conjugate to any of the
  groups $A_4, S_4$ or $A_5$. Thus $\rho(Aut(C))$ is conjugate to $D_{2(d-2)}$ but also we have $4(d-2)||Aut(C)|$ therefore $2||N|$
  and the case $(5)$ is proved.} It remains now to {determine the full automorphism group} when $d=6$ or $10${:}
\par For $d=6$, {the equation $(\ref{(Eq.1)})$ in Proposition \ref{prop111} become} $X^{6}+Y^{5}Z+YZ^{5}+{\beta_{2,4}}X^{2} Y^{2} Z^{2}=0$
with $\beta_{2,4}\neq 0$. {Let $\eta\in Aut(C)$ then $\eta$ is of the form $[X;\beta_2Y;\gamma_3Z]$ or
 $[X;\beta_3Z;\gamma_2Y]$, since} the monomials $X^2Y^4$ and $X^2Z^4$ {are not in the defining equation of $C$}.
  Hence {we must have} $\beta_2^5\gamma_3=\beta_2\gamma_3^5=\beta_2^2\gamma_3^2=1,$ which in turns implies that $|Aut(C)|=16$.
   Therefore $Aut(C)$ is conjugate to $<\sigma,\tau|\tau^2=\sigma^{8}=1\, and\,\tau\sigma\tau=\sigma^{3}>$ with
    $\sigma:=[X;\xi_8Y;\xi_8^3Z]$ and $\tau:=[X;Z;Y]$ which is $SmallGroup(16,8)$ in Gap list. By a {quite} similar argument, {one}
     conclude that {when} $d=10$, the plane non-singular model is reduced to
     $X^{10}+Y^{9}Z+YZ^{9}+{\beta_{6,4}}X^{6} Y^{2} Z^{2}+{\beta_{2,8}}X^{2} Y^{4} Z^{4}$ with $(\beta_{6,4},\beta_{2,8})\neq (0,0)$.
     {Also} $|Aut(C)|=32$ where $Aut(C)=<\sigma,\tau>$ with $\tau:=[X;Z;Y]$ and $\sigma:=[X;\xi_{16}Y;\xi_{16}^{-9}Z]$
     and hence $Aut(C)$ is isomorphic to $SmallGroup(32,19)$.\\
This completes the proof.
\end{proof}

\begin{cor} {The locus $\widetilde{M_g^{Pl}(\Z/\ell(d-2)\Z)}$ is always empty for any even integer $\ell\geq2$.}
\end{cor}

%\noindent In what follows, we study the loci $M_g^{Pl}(\Z/\ell(d-2)\Z)$ for any arbitrary but a fixed degree $d\geq6$ such that $2\nmid \ell$.
{Now we treat the situation for which $\ell$ is odd:}
\begin{prop}\label{P2}
{Suppose that $\ell\geq2$ is an odd integer such that $\ell|d$ with
$d\geq6$. Any non-singular} plane model of $\delta\in
M_g^{Pl}(\Z/\ell(d-2)\Z)$ {is $K$-isomorphic to the form}
\begin{equation}\label{(Eq.2)}
X^d+Y^{d-1}Z+\alpha YZ^{d-1}+\sum_{k=1}^{n}\,\,\beta_{2\ell k,\ell
k}X^{d-2\ell k}Y^{\ell k}Z^{\ell k}=0,
\end{equation}
{where $n=\frac{d}{2\ell}$ if $d$ is even and
$\lfloor\frac{d-1}{2\ell}\rfloor$ otherwise.} In particular, the
loci $M_g^{Pl}(\Z/\ell(d-2)\Z)$ are ES-irreducible.
\end{prop}
\begin{proof}
{Again,} by Lemma \ref{E1}, any plane non-singular model of $\delta$
is $K$-isomorphic to type $\ell(d-2), (a,b)$ {of the form $(4.1)$ of
Theorem \ref{thm20}} for some $(a,b)\in \Gamma_{\ell(d-2)}$ and
$\ell(d-1)|(d-1)a+b,\,a+(d-1)b$. {In particular,} $2a=(d-2)k'_0+\ell
k_0$ and $2b=(d-2)k'_0-\ell k_0$ for some integers $k_0$ and $k'_0$
and we distinguish between whether $d$ is even or odd as follows: If
$d$ is even then so is $k_0$ and $a=\ell
k+(\frac{d-2}{2})k',\,b=-\ell k+(\frac{d-2}{2})k'$ for some integers
$k$ and $k'$. {Moreover $\ell|\frac{d}{2}k'$, since
$\ell(d-2)|(d-1)a+b$ and consequently}
$[X;\xi_{\ell(d-2)}Y;\xi_{\ell(d-2)}^{(\ell-1)(d-2)-1}Z]^{\ell
k+(\frac{d-2}{2})k'}=[X;\xi_{\ell(d-2)}^aY;\xi_{\ell(d-2)}^{b}Z].$
Therefore $a=1$ and $b=(\ell-1)(d-2)-1$ is a generator of the set of
solutions of the system. {As usual, it remains to determine the sets
$S_2^{d,X}$ and $S(2)^{j,X}$} for $j=2,...,d-1$ with
$m=\ell(d-2),\,a=1$ and $b=(\ell-1)(d-2)-1$. {In fact these sets are
the same as seen in the proof of Proposition \ref{P1} and the rest
will be typical except possibly we use the automorphism
$[X;\xi_{\ell(d-2)}Y;\xi_{\ell(d-2)}^{-(d-1)}Z]$ instead of
$[X;\xi_{\ell(d-2)}Y;\xi_{\ell(d-2)}^{d-3}Z]$ to obtain the required
equation in this case.}
%\begin{eqnarray*}
%{S_2^{d,X}}&:=&\{i:\,\,2\leq i\leq d-2\,\,\text{and}\,\,\ell(d-2)|i+(d-i)\left((\ell-1)(d-2)-1\right)\}\\
%&=&\{i:\,\,2\leq i\leq d-2\,\,\text{and}\,\,d-2|\frac{d}{\ell}(i-1)\}\subseteq\{\dfrac{d}{2}\}
%\end{eqnarray*}
%The last inclusion follows because {$0<\leq\frac{d}{\ell}(i-1)<\frac{d}{\ell}(d-2)$. That is} $\frac{d}{\ell}(i-1)=\mu(d-2)$ for some $1\leq\mu\leq\frac{d}{\ell}-1$ {hence} $\frac{d}{\ell}|2\mu$ and $\mu=\frac{d}{2\ell}$. Also
%\begin{eqnarray*}
%{S(2)^{j,X}}&:=&\{i:\,\,0\leq i\leq j\,\,\text{and}\,\,\ell(d-2)|i+(j-i)\left((\ell-1)(d-2)-1\right)\}\\
%&=&\{i:\,\,0\leq i\leq j\,\,\text{and}\,\,\ell(d-2)|(d-1)j-di\}\\
%&\subseteq&\{i:\,\,0\leq i\leq j\,\,\text{and}\,\,(d-2)|j-2i\}
%\end{eqnarray*}
%But we have $|j-2i|d-1$ then by the same argument as before the only possibilities are $j-2i=0,\pm(d-2)$. In particular, $S(2)^{j,X}_{m,(a,b)}$ is empty if $j$ is odd, $\{\frac{j}{2}\}$ if $j<d-2$ and even and $\{0,\frac{d-2}{2},d-2\}$ otherwise. Consequently, we obtain the form
%$$
%X^d+Y^{d-1}Z+\alpha YZ^{d-1}+X^{2}\big(\beta_{d-2,0}Z^{d-2}+\beta_{0,d-2}Y^{d-2}\big)+\sum_{j=2,4,...,d-2,d}\,\,\beta_{j,\frac{j}{2}}X^{d-j}Y^{\frac{j}{2}}Z^{\frac{j}{2}}=0
%$$
%Because we have $[X;\xi_{\ell(d-2)}Y;\xi_{\ell(d-2)}^{-(d-1)}Z]\in Aut(C)$ then $\beta_{d-2,0}=\beta_{0,d-2}=0$ moreover if $2\ell\nmid j$ then $\beta_{j(\frac{j}{2})}=0$.
If $d$ is odd then $k_0$ and $k'_0$ have the same parity and
$a={\frac{1}{2}(\ell k_0+k'_0(d-2))}, b={\frac{1}{2}(-\ell
k_0+k'_0(d-2))}$. {Also} $2|\pm k_0+(\frac{d}{\ell})k'_0$, since
$\ell(d-2)|(d-1)a+b,\,a+(d-1)b$ {and} in particular, we can replace
$k_0$ by $2k-(\frac{d}{\ell})k'_0$ for some integer $k$.
Consequently $\xi_{\ell(d-2)}^b=\xi_{\ell(d-2)}^{-(d-1)a}$ and
$[X;\xi_{\ell(d-2)}Y;\xi_{\ell(d-2)}^{-(d-1)}Z]^{a}=[X;\xi_{\ell(d-2)}^aY;\xi_{\ell(d-2)}^{b}Z]$.
Hence $a=1$ and $b=(\ell-1)(d-2)-1$ is again a generator of the set
of solutions. Finally, the sets ${S_2^{d,X}}$ and ${S(2)^{j,X}}$ for
$j=2,...,d-1$ with $m=\ell(d-2)$, $a=1$ and $b=(\ell-1)(d-2)-1$ {are
given below:}
\begin{eqnarray*}
{S_2^{d,X}}&:=&\{i:\,\,2\leq i\leq d-2\,\,\text{and}\,\,\ell(d-2)|i+(d-i)\left((\ell-1)(d-2)-1\right)\}\\
&=&\{i:\,\,2\leq i\leq d-2\,\,\text{and}\,\,d-2|\frac{d}{\ell}(i-1)\}\\
&=&\emptyset
\end{eqnarray*}
The last inclusion {can be easily deduced because}
${0<\frac{d}{\ell}(i-1)<\frac{d}{\ell}(d-2)}$. Therefore
$\frac{d}{\ell}(i-1)=\mu(d-2)$ for some
$1\leq\mu\leq\frac{d}{\ell}-1$. {This in turns gives}
$\frac{d}{\ell}|\mu$ (since $\frac{d}{\ell}$ is odd) which is not
possible. Also, {we have}
\begin{eqnarray*}
{S(2)^{j,X}}&:=&\{i:\,\,0\leq i\leq j\,\,\text{and}\,\,\ell(d-2)|i+(j-i)\left((\ell-1)(d-2)-1\right)\}\\
&=&\{i:\,\,0\leq i\leq j\,\,\text{and}\,\,\ell(d-2)|(d-1)j-di\}\\
&\subseteq&\{i:\,\,0\leq i\leq j\,\,\text{and}\,\,(d-2)|j-2i\}
\end{eqnarray*}
Because {$|j-2i|\leq d-1$ therefore $j-2i=0,\pm (d-2)$ and} $
S(2)^{j,X}_{m,(a,b)}=\left\{
\begin{array}
[c]{lr}%
\emptyset,\,\,\,\,\,\,\,\,\,\,\,\,\,\,\,\,\,\,\,\,\,\,\,\,\,\,\,\,\,\,\,\,\,\,\,\,\,\,\,\,\,if\,\,j\in\{1,3,...,d-4\}& \\
\{0,d-2\},\,\,\,\,\,\,\,\,\,\,\,\,\,\,\,\,\,\,if\,j=d-2 &\\
\{\frac{j}{2}\}\,\,\,\,\,\,\,\,\,\,\,\,\,\,\,\,\,\,\,\,\,\,\,\,\,\,\,\,\,otherwise
\end{array}
\right.$
 Moreover we obtain the form  $$X^d+Y^{d-1}Z+\alpha YZ^{d-1}+X^{2}\big(\beta_{d-2,0}Z^{d-2}
 +\beta_{d-2,d-2}Y^{d-2}\big)+\sum_{j=2,4,...,d-1}\,\,\beta_{j,\frac{j}{2}}X^{d-j}Y^{\frac{j}{2}}Z^{\frac{j}{2}}
=0$$
But $[X;\xi_{\ell(d-2)}Y;\xi_{\ell(d-2)}^{-(d-1)}Z]\in Aut(C)$ then $\beta_{d-2,0}=\beta_{d-2,d-2}=0$ moreover $\beta_{j,\frac{j}{2}}=0$ if $\ell\nmid \frac{j}{2}$.\\
This completes the proof.
\end{proof}
%The following result determines a necessary condition for which the locus $\widetilde{M_g^{Pl}(G)}$ is non-empty where $G$ is a finite group inside $PGl_3(K)$ such that $\Z/\ell(d-2)\Z\preceq G$ where $\ell>3$ and $2\nmid \ell$.
{The full automorphism group of the elements of the locus
$M_g^{Pl}(\Z/\ell(d-2)\Z)$ with $\ell\geq3$ odd is determined by the
result:}
\begin{prop}
{Let $\ell\geq3$ be an odd integer such that $\ell|d$ with $d\geq6$
and let $G$ be a finite group inside $PGL_3(K)$. Then $\delta\in
M_g^{Pl}(\Z/\ell(d-2)\Z)\cap \widetilde{M_g^{Pl}(G)}$} only if one
of the following situations occurs:
\begin{enumerate}
\item {$d=6$ and $G$ is conjugate to a central extension of $S_4$ by $\Z/6\Z$.
In this case, $G$ is of order $144$ and $\widetilde{M_g^{Pl}(G)}$ is an irreducible set that is given by the single element $X^6+Y^{5}Z+YZ^{5}=0$.}
\item {$d>6$ and $G$ is conjugate to $<\sigma,\tau|\tau^2=\sigma^{d(d-2)}=1,\,\tau\sigma\tau=\sigma^{-(d-1)}>$,
a central extension of order $2d(d-2)$ of $D_{2(d-2)}$ by $\Z/d\Z$.
Also $\widetilde{M_g^{Pl}(G)}$ is an irreducible set and is given by
one element with a non-singular plane model isomorphic to
$X^d+Y^{d-1}Z+YZ^{d-1}=0$.}

\item $\ell=5, d=10$ and $G$ is conjugate to $SmallGroup(80,25)$. {In this case every}
$\delta\in M_{36}^{Pl}(SmallGroup(80,25))$ has a {non-singular
plane} model {which is $K$-equivalent} to
$X^{10}+Y^9Z+YZ^9+\beta_{10,5}Y^5Z^5=0$ with $\beta_{10,5}\neq0$.
\item {$\ell>3,\,d\neq10$ and $G$ is an element of $Ext^1(N,D_{2m})$} where $N$ is a cyclic group order dividing $d$ {and $m=d-2$
with $2\nmid d$ and $\ell|\,|N|$ or $m=\frac{d-2}{2}$ with $2|d$ and
$2\ell|\,|N|$. Moreover $G$ contains a subgroup which is isomorphic
to}
$<\sigma,\tau:\,\tau^2=\sigma^{\ell(d-2)}=1\,\textit{and}\,\tau\sigma\tau=\sigma^{-(d-1)}>$
{as a subgroup. Also, every element} $\delta\in
{\widetilde{M_g^{Pl}(G)}}$ has a plane model {that is $K$-
equivalent to $(\ref{(Eq.2)})$} such that $\beta_{2\ell j,\ell j}
\neq0$ for some $j\in\{1,2,...,n\}$.
\end{enumerate}
\end{prop}

\begin{proof}
{We could apply the same argument of Proposition \ref{P3} to
conclude the following:}
\begin{itemize}
  \item {Case $(1)$ or $(2)$ occurs if and only if $\beta_{2\ell k,\ell k}=0$ for all $k\in\{1,2,...,n\}$.}
  \item {Every plane non-singular model $C$ of $\delta\in M_g^{Pl}(\Z/\ell(d-2)\Z)$ (which is isomorphic to equation $(\ref{(Eq.2)})$) admits always $G_0$ as a subgroup of order $2\ell(d-2)$ with $\sigma:=[X;\xi_{\ell(d-2)}Y;\xi_{\ell(d-2)}^{-(d-1)}Z]$ and $\tau:=[X;\mu Z;\mu^{-1}Y]$ where $\mu^{d-2}=\alpha.$ In particular, $Aut(\delta)$ is not cyclic}
  \item {$\delta$ is not a descendant of the the Klein curve and also $Aut(\delta)$ is not conjugate to any of the finite primitive groups inside $PGL_3(K)$.}
  \item {If $\ell\neq3$ or $d\neq6$, $\delta$ is not a descendant of the Fermat curve.}
\end{itemize}
{Assuming that $\ell\neq3$ or $d\neq6$ and following the same
ideas,} we can think about $Aut(C)$ in a short exact sequence
$1\rightarrow N\rightarrow Aut(C)\rightarrow \rho(Aut(C))\rightarrow
1$ where $\rho(Aut(C))$ contains the element $\rho(\tau)=\left(
  \begin{array}{cc}
0 & \mu \\
 \mu^{-1} & 0 \\
\end{array}
\right)$ of order $2$ and the element $\rho(\sigma)=\left(
\begin{array}{cc}
1 & 0 \\
0 & \xi_{d-2}^{\frac{d}{\ell}} \\
 \end{array}
 \right)$ of order {$d-2$ (only if $2\nmid d$) and $\frac{d-2}{2}$ (otherwise)}. In particular,
 $\rho(Aut(C))$ is not cyclic. Moreover, if $2\nmid d$ {(resp. $2|d$ and $d\neq10$)}
 then $\rho(Aut(C))$ is not conjugate to $A_4, S_4$ or $A_5$
 {, since it has an element of order $>5$}. Consequently $\rho(Aut(C))$ is conjugate to $D_{2(d-2)}$ {or $D_{2(\frac{d-2}{2})}$ (only if $2|d$)} and $|Aut(C)|=2(d-2)||N|$ or $(d-2)|N|$. {Therefore $|N|$ should be divisible by $\ell$ or $2\ell$,} since $2\ell(d-2)||Aut(C)|$.
\par If $d=10$ {then} $\ell=5$ {and} the equation $(\ref{(Eq.2)})$ in
Proposition \ref{P2} is reduced to
$X^{10}+Y^9Z+YZ^9+\beta_1Y^5Z^5=0$. {Also} $N=<diag(\xi_{10};1;1)>$
{and} $\rho(Aut(C))$ is not conjugate to $A_4$ or $A_5${, since}
$D_{8}\preceq\rho(Aut(C))$. {Therefore $\rho(Aut(C))$ is conjugate
to $S_4, D_{16}$ or $D_8$}. {If $\rho(Aut(C))\equiv S_4$ then there
exists an element $\tau'\in PGL_2(K)$ such that
$\rho(\sigma)^2\tau'$ and $\tau'^{-1}\rho(\tau)\rho(\sigma)^2$ are
of order 2 and moreover
$\rho(\sigma)^2(\rho(\tau)\tau'\rho(\tau))=\rho(\tau)\rho(\sigma)^2\tau'$.
The first relation gives $\tau'=\left(
\begin{array}{cc}
\mu_1 & \xi_4^{a} \\
1 & \mu_1 \\
\end{array}
\right)$,
% In fact because $\rho(\sigma)^2\tau'$ has order 2 then we have two possibilities $\tau'=\left(\begin{array}{cc}
%\mu_1 & 0 \\
%0 & \mu_4 \\
%\end{array}
%\right)}$ (being of order 3 gives $\mu_1=\xi_3^{b}\mu_4$ and thus does not preserve $Y^9Z+YZ^9+...$ a contradiction) or $\mu_1=\mu_4$ (note that because $\alpha=1$ then $\mu=\xi_8^a$). Secondly because $\tau'^{-1}\rho(\tau)\rho(\sigma)^2$ has order 2, then we get
%$\tau'=\left(\begin{array}{cc}
%\mu_1 & \xi_4^{a}\mu_3 \\
%\mu_3 & \mu_1 \\
%\end{array}
%\right)}$ and being of order 3 implies $\mu_3\neq0$ that is $\tau'$ is equivalent to the prescribed form.
and then imposing the second condition to get $\exists \lambda\in K^*$ such that $\lambda \mu=-\mu_1,\,\lambda\mu_1=-\mu,\,\lambda\mu_1=\mu^{-1}\xi_4^{a}$ and $\lambda\mu^{-1}\xi_4^{a}=\mu_1$ hence $-1=\lambda^2=1$ a contradiction.
If $\rho(Aut(C))\equiv D_{16}$ then there must be an element
$\tau'\in PGL_2(K)$ of order $2$ such that $\tau'\rho(\sigma)^2$ has order $8$ with $\rho(\tau),\rho(\sigma)\in <\tau',\rho(\sigma)^2>=D_{16}$. In particular $(\tau'\rho(\sigma)^2)^2=\rho(\sigma)$ or $\rho(\sigma)^{-1}$ (being the only elements of order $4$ inside $D_{16}$) hence
%$:=\left(
%\begin{array}{cc}
%\mu_1 & \mu_2 \\
%\mu_3 & \mu_4 \\
%\end{array}
%\right)
%$
%In particular $\mu_4=-\mu_1$ (being of order $2$). Moreover if $\tau'\rho(\sigma)^2$ has order $8$ then $\mu_1=0$, since $D_{16}$ has only two elements of order $4$ which are $\rho(\sigma)$ and $\rho(\sigma)^{-1}$.
$\tau'=\left(
\begin{array}{cc}
0 & \mu_2 \\
\mu_3 & 0 \\
\end{array}
\right)$. In this case $\tau'\tau'\rho(\sigma)^2$ is of order $2<8$
a contradiction. } {We then conclude that} $\rho(Aut(C))$ is
conjugate to $D_{8}$ and $|Aut(C)|=80$. More precisely, $Aut(C)$ is
generated by $\sigma:=[X;\xi_{40}Y;\xi_{40}^{-9}Z]$ and
$\tau:=[X;Z;Y]$ {which is} isomorphic to
$<\sigma,\tau:\,\tau^2=\sigma^{40}=1\,\textit{and}\,\tau\sigma\tau=\sigma^{-9}>{\cong
SmallGroup(80,25)}$.
\par {Finally it remains to treat the case $\ell=3$ and $d=6$ where $C:\,X^{6}+Y^5Z+YZ^5+\beta_1Y^3Z^3=0$ is a descendant of the Fermat sextic curve through a transformation $P\in PGL_3(K)$. Since $C$ admits an automorphism $\sigma:=[X;\xi_{12}Y;\xi_{12}^{-5}Z]$ of order 12 then $\sigma^4=[\omega X;Y;Z]\in Aut(C)$ is a homology of order $3$. Also homologies of order 3 inside $Aut(F_6)$ are divided into $S_1:=\{[\omega X;Y;Z],\,[X;\omega Y;Z],\,[X;Y;\omega Z]\}$ and
$S_2:=\{[\omega^2 X;Y;Z],\,[X;\omega^2 Y;Z],\, [X;Y;\omega^2 Z]\}$
where both sets lie in different conjugacy classes in $PGL_3(K)$.
Consequently ${P^{-1}\sigma^4P}\in S_1$ and because the elements of
$S_1$ are conjugate to each others inside $Aut(F_6)$, we need only
to consider the situation ${P^{-1}\sigma^4P}=\sigma^4$. Thus
$P=[X;\mu_2Y+\mu_3Z;\gamma_2Y+\gamma_3Z]$ and $C$ is transformed to
the form
%and elements of order $12$ in $Aut(F_6)$ are $[X;\xi_6^aZ;\xi_6^bY],\,[\xi_6^bY;\xi_6^aX;Z]$ or $[\xi_6^bZ;Y;\xi_6^aX]$ such that $gcd(6,a+b)=1$. We can assume, without loss of generality, that $a=0$ and $b=1$ or $5$, since any other element of order $12$ is conjugate to one of these inside $Aut(F_6)$. Now it follows by our assumption that $\exists P\in PGL_3(K)$ with $P\sigma P^{-1}=[X;Z;\xi_6Y]$ or $[X;Z;\xi_6^{-1}Y]$
$\widehat{C}:\,\,X^6+\nu_0Y^6+\nu_1Z^6+G(Y,Z)$
where $\nu_0:=\gamma _2 \mu _2 \left(\gamma _2^4+\beta  \mu _2^2
   \gamma _2^2+\mu _2^4\right)(=1)$ and $\nu_1:=\gamma _3 \mu _3 \left(\gamma _3^4+\beta  \mu _3^2
   \gamma _3^2+\mu _3^4\right)(=1)$. In particular,
$(\gamma _2 \mu _2)(\gamma _3 \mu _3)\neq0$ and
$[\xi_6^bY;\xi_6^aX;Z],\,[\xi_6^bZ;Y;\xi_6^aX]\notin
Aut(\widehat{C})$. Hence ${P^{-1}\sigma P}=[X;Z;\xi_6^bY]\in
Aut(\widehat{C})$ with $b=1$ or $5$, since elements of order $12$ in
$Aut(F_6)$ are $[X;\xi_6^aZ;\xi_6^bY],\,[\xi_6^bY;\xi_6^aX;Z]$ or
$[\xi_6^bZ;Y;\xi_6^aX]$ such that $gcd(6,a+b)=1$ and moreover any
such element is conjugate inside $Aut(F_6)$ to $[X;Z;\xi_6^bY]$ with
$b=1$ or $5$. On the other hand, ${P^{-1}\tau P}\in
Aut(\widehat{C})$ is of order $2$
% we assume $\mu=1$ because already $\alpha=1$
thus $\mu_3=\mu_2,\,\gamma_3=\gamma_2$ or $\mu_3=-\mu_2,\,\gamma_3=-\gamma_2$, which in turns reduces $\widehat{C}$ to  $X^6+(Y\pm Z)^6$. This is not possible because $[X;Z;\xi_6^bY]$ with $b=1$ or $5$ does not retain $\widehat{C}$, therefore $C$ is not be a descendant of the Fermat curve.
}
\\
This completes the proof.
\end{proof}
%%%%%%%%%%%%%%%%%%%%%%%%%%%%%%%%%%%%%%%%%%%%%%%%%%%%%%%
%%%%%%%%%%%%%%%%%%%%%%%%%%%%%%%%%%%%%%%%%%%%%%%%%%%%%%%%%
%\vspace{-1.5cm}
\appendix{}

\section{Tables of Type $m(a,b)$ for degree $d\leq 9$}
In this appendix we introduce tables {for the} types of cyclic
groups and {the} equations that are
obtained {as} a result of \S 2 {with respect to low} degrees{. I}n particular, we list {the
possible} $m,(a,b)$ such that $\rho_{m,a,b}(M_g^{Pl}(\Z/m))$ may be
non-trivial, and {we associate a normal form} $F(X;Y;Z)=0$ for such
loci, {where} any element of the {locus} has a plane non-singular
model {for some} specialization of the parameters. The notation of
the parameters, for a fixed degree $d$, are unrelated from one type
to another one: for example, we use, by an abuse of notation,
$\beta_{i,j}$ as the parameter of the monomial $X^{d-j}Y^iZ^{j-i}$
in any normal form.

It might happen that two types $m, (a,b)$ and $m, (a',b')$ are
isomorphic through a permutation of the variables or $F(X;Y;Z)$
decomposes into a product $X.G(X;Y;Z)$. The following tables are
obtained by {compiling the SAGE code of Theorem \ref{thm20}} and
then removing those types which are isomorphic to a certain type or
are not irreducible, see the programm in
http://mat.uab.cat/$\sim$eslam/CAGPC.sagews
\begin{small}
\begin{center}
\begin{table}[!th]
  \renewcommand{\arraystretch}{1.3}
  \caption{Quartics\,\,\,}\label{table:Cyclic Auto42.}
  \vspace{4mm} % hack
  \centering
\begin{tabular}{|c|c|}
  \hline
  Type: $m, (a,b)$ & $F(X;Y;Z)$ \\\hline\hline
  $12,(3,4)$& $X^4+Y^4+\alpha XZ^3$ \\\hline
  $9,(1,6)$& $X^4+Y^3Z+\alpha XZ^3$ \\\hline
  $8,(1,5)$& $X^4+Y^3Z+\alpha YZ^3$ \\\hline
  $7,(1,5)$& $X^3Y+Y^3Z+\alpha Z^3X$ \\\hline
$6,(3,4)$ & $X^4+Y^4+\alpha XZ^3+\beta_{2,2}X^2Y^2$ \\\hline
 $4,(1,2)$& $X^4+Y^4+Z^4+\beta_{2,0}X^2Z^2+\beta_{3,2}XY^2Z $ \\\hline
% $\textbf{4,(2,3)}$& $X^4+Y^4+Z^4+\beta_{2,2}X^2Y^2+\beta_{3,1}XYZ^2 $ \\\hline
 %$\textbf{4,(1,3)}$& $X^4+Y^4+Z^4+\beta_{4,2}Y^2Z^2+\beta_{2,1}X^2YZ $ \\\hline
 $4,(0,1)$& $Z^4+L_{4,Z}$ \\\hline
%$3,(1,2)$& $X^4+X\big(Z^3+\alpha Y^3\big)+\beta_{2,1}X^2YZ+\beta_{4,2}Y^2Z^2$ \\\hline
 % $3,(0,1)$& $Z^3L_{1,Z}+L_{4,Z}$ \\\hline
  %$2,(0,1)$& $Z^4+Z^2L_{2,Z}+L_{4,Z}$ \\\hline
 % \end{tabular}
%\end{table}
%\end{center}
%\end{small}
%
%\begin{small}
%\begin{center}
%\begin{tabular}{|c|c|}
%  \hline
%  Type: $m, (a,b)$ & $F(X;Y;Z)$ \\\hline\hline
 % $12,(3,4)$& $X^4+Y^4+\alpha XZ^3$ \\\hline
 % $9,(1,6)$& $X^4+Y^3Z+\alpha XZ^3$ \\\hline
 % $8,(1,5)$& $X^4+Y^3Z+\alpha YZ^3$ \\\hline
  %$7,(1,5)$& $X^3Y+Y^3Z+\alpha Z^3X$ \\\hline
%$6,(3,4)$ & $X^4+Y^4+\alpha XZ^3+\beta_{2,2}X^2Y^2$ \\\hline
% $4,(1,2)$& $X^4+Y^4+Z^4+\beta_{2,0}X^2Z^2+\beta_{3,2}XY^2Z $ \\\hline
% $\textbf{4,(2,3)}$& $X^4+Y^4+Z^4+\beta_{2,2}X^2Y^2+\beta_{3,1}XYZ^2 $ \\\hline
 %$\textbf{4,(1,3)}$& $X^4+Y^4+Z^4+\beta_{4,2}Y^2Z^2+\beta_{2,1}X^2YZ $ \\\hline
% $4,(0,1)$& $Z^4+L_{4,Z}$ \\\hline
$3,(1,2)$& $X^4+X\big(Z^3+\alpha Y^3\big)+\beta_{2,1}X^2YZ+\beta_{4,2}Y^2Z^2$ \\\hline
  $3,(0,1)$& $Z^3L_{1,Z}+L_{4,Z}$ \\\hline
  $2,(0,1)$& $Z^4+Z^2L_{2,Z}+L_{4,Z}$ \\\hline
  \end{tabular}
  \end{table}
\end{center}
\end{small}

\begin{center}
\begin{table}[!th]
  \renewcommand{\arraystretch}{1.3}
  \caption{Quintics\,\,\,}\label{table:Cyclic Auto52.}
  \vspace{4mm} % hack
  \centering
\begin{tabular}{|c|c|}
  \hline
  Type: $m, (a,\,b)$ & $F(X;Y;Z)$ \\\hline\hline
  $20,(4,5)$& $X^5+Y^5+\alpha XZ^4$ \\\hline
  $16,(1,12)$& $X^5+Y^4Z+\alpha XZ^4$\\\hline
  $15,(1,11)$& $X^5+Y^4Z+\alpha YZ^4$    \\\hline
$13,(1,10)$& $X^4Y+Y^4Z+\alpha Z^4X$    \\\hline
 $10,(2,5)$& $X^5+Y^5+\alpha XZ^4+\beta_{2,0}X^3Z^2$ \\\hline
  $8,(1,4)$& $X^5+Y^4Z+\alpha XZ^4+\beta_{2,0}X^3Z^2$ \\\hline
 $5,(1,2)$& $X^5+Y^5+Z^5+\beta_{3,1}X^2YZ^2+\beta_{4,3}XY^3Z$    \\\hline
%$\textbf{5,(1,3)}$& $X^5+Y^5+Z^5+\beta_{3,2}X^2Y^2Z+\beta_{4,1}XYZ^3$    \\\hline
%$\textbf{5,(1,4)}$& $X^5+Y^5+Z^5+\beta_{2,1}X^3YZ+\beta_{4,2}XY^2Z^2$    \\\hline
  $5,(0,1)$& $Z^5+L_{5,Z}$    \\\hline
 % \end{tabular}
%\end{table}
%\end{center}
%
%
%\begin{center}
%\begin{tabular}{|c|c|}
%  \hline
 $4,(1,2)$& $X^5+X\big(Z^4+\alpha Y^4\big)+\beta_{2,0}X^3Z^2+\beta_{3,2}X^2Y^2Z+\beta_{5,2}Y^2Z^3$\\\hline
%$\textbf{4,(1,3)}$& $X^5+X\big(Z^4+\alpha Y^4\big)+\beta_{2,1}X^3YZ+\beta_{4,2}XY^2Z^2$\\\hline
 %$\textbf{4,(2,3)}$& $X^5+X\big(Z^4+\alpha Y^4\big)+\beta_{2,2}X^3Y^2+\beta_{3,1}X^2YZ^2+\beta_{5,3}Y^3Z^2$\\\hline

 \end{tabular}
\end{table}
\end{center}

\begin{small}
\begin{center}
\begin{tabular}{|c|c|}
  \hline
  %Type: $m, (a,b)$ & $F(X;Y;Z)$ \\\hline\hline

  $4,(0,1)$& $Z^4L_{1,Z}+L_{5,Z}$    \\\hline
  $3,(1,2)$& $X^5+Y^4Z+\alpha YZ^4+\beta_{2,1}X^3YZ+X^2\big(\beta_{3,0}Z^3+\beta_{3,3}Y^3\big)+\beta_{4,2}XY^2Z^2$ \\\hline
 $2,(0,1)$& $Z^4L_{1,Z}+Z^2L_{3,Z}+L_{5,Z}$    \\\hline
  \end{tabular}

\end{center}

\end{small}

%\newpage

\begin{small}
\begin{center}
\begin{table}[!th]
  \renewcommand{\arraystretch}{1.3}
  \caption{Sextics\,\,\,}\label{table:Cyclic Auto62.}
  \vspace{4mm} % hack
  \centering
\begin{tabular}{|c|c|}
  \hline
  Type: $m, (a,\,b)$ & $F(X;Y;Z)$ \\\hline\hline
$30,(5,6)$& $X^{6} + Y^{6}+\alpha X Z^{5} $ \\\hline
$25,(1,20)$& $X^{6} + Y^{5} Z+\alpha X Z^{5} $ \\\hline
$24,(1,19)$& $X^{6} + Y^{5} Z +\alpha Y Z^{5}$ \\\hline
  $21,(1,17)$& $X^{5} Y + Y^{5} Z+\alpha X Z^{5} $ \\\hline
$15,(5,6)$& $X^{6} + Y^{6}+\alpha X Z^{5}+{\beta_{3,3}} X^{3} Y^{3} $ \\\hline
 $12,(1,7)$& $X^{6} + Y^{5} Z +\alpha Y Z^{5} +{\beta_{6,3}}Y^{3} Z^{3} $ \\\hline
$10,(5,6)$& $X^{6} + Y^{6}+\alpha X Z^{5}+ {\beta_{2,2}}X^{4} Y^{2} + {\beta_{4,4}}X^{2} Y^{4}  $ \\\hline
 $8,(1,3)$& $X^{6} + Y^{5} Z +\alpha Y Z^{5} + {\beta_{4,2}}X^{2} Y^{2} Z^{2} $ \\\hline
$6,(1,2)$& $X^{6} + Y^{6} + Z^{6}+{\beta_{3,0}}X^{3} Z^{3}  + {\beta_{4,2}}X^{2} Y^{2} Z^{2}  + {\beta_{5,4}}X Y^{4} Z $ \\\hline
$6,(1,3)$& $X^{6} + Y^{6} + Z^{6}+{\beta_{2,0}}X^{4} Z^{2}  + {\beta_{6,3}} Y^{3} Z^{3} +  X^{2}{\left({\beta_{4,0}}Z^{4} + {\beta_{4,3}}Y^{3} Z \right)}$ \\\hline
%$\textbf{6,(1,4)}$& $X^{6}+Y^{6}+Z^{6}+\beta_{6,2}Y^{2} Z^{4}+\beta_{6,4}Y^{4} Z^{2}+X^{3}\left(\beta_{3,0}Z^{3}+ \beta_{3,2}Y^{2}Z \right) $ \\\hline
%$\textbf{6,(1,5)}$& $X^{6}+Y^{6}+Z^{6}+\beta_{2,1}X^{4}Y Z+\beta_{4,2}X^{2}Y^{2}Z^{2}+\beta_{6,3}Y^{3} Z^{3}$\\\hline
%$\textbf{6,(2,3)}$& $X^{6}+Y^{6}+Z^{6}+\beta_{2,0}X^{4} Z^{2}+ {{\beta_{3,3}}} X^{3} Y^{3}+{{\beta_{4,0}}} X^{2} Z^{4}+{{\beta_{5,3}}} X Y^{3} Z^{2}$\\\hline
%$\textbf{6,(2,5)}$& $X^{6} + Y^{6} + Z^{6}+{{\beta_{6,2}}}Y^{2} Z^{4} +{{\beta_{6,4}}}Y^{4} Z^{2}+X^{3}\left({{\beta_{3,1}}}Y Z^{2}+ {{\beta_{3,3}}}Y^{3} \right)$ \\\hline
%$\textbf{6,(3,4)}$& $X^{6}+Y^{6}+Z^{6}+ \beta_{2,2}X^{4}Y^{2}+ \beta_{3,0}X^{3}Z^{3}+ \beta_{4,4}X^{2}Y^{4}+\beta_{5,2}XY^{2}Y^{3}$ \\\hline
%$\textbf{6,(3,5)}$ & $X^{6}+Y^{6}+Z^{6}+\beta_{2,2}X^{4} Y^{2}+\beta_{6,3}Y^{3}Z^{3} + X^{2}\left(\beta_{4,1}YZ^{3}+\beta_{4,4}Y^{4}\right)$\\\hline
%$\textbf{6,(4,5)}$& $X^{6}+Y^{6}+Z^{6}+\beta_{3,3}X^{3}Y^{3}+\beta_{4,2} X^{2}Y^{2} Z^{2}+ \beta_{5,1}X Y Z^{4}$\\\hline
$6,(0,1)$& $Z^{6} + {L_{6,Z}}$ \\\hline
$5,(1,2)$& $X^{6} + X Z^{5}+\alpha X Y^{5}  + {\beta_{3,1}}X^{3} Y Z^{2}  + {\beta_{4,3}}X^{2} Y^{3} Z  +{\beta_{6,2}} Y^{2} Z^{4} $ \\\hline
%$\textbf{5,(1,3)}$& $X^{6} + X Z^{5}+\alpha X Y^{5}  + {\beta_{3,2}}X^{3} Y^{2} Z  + {\beta_{4,1}}X^{2} Y Z^{3}  +{\beta_{6,4}} Y^{4} Z^{2} $ \\\hline
$5,(1,4)$& $X^{6} + X Z^{5}+\alpha X Y^{5}+{\beta_{2,1}}X^{4} Y Z  + {\beta_{4,2}}X^{2} Y^{2} Z^{2}  + {\beta_{6,3}}Y^{3} Z^{3} $ \\\hline
  $5,(0,1)$& $Z^{5}{L_{1,Z}}  + {L_{6,Z}}$ \\\hline
$4,(1,3)$& $X^{6} + Y^{5} Z +\alpha Y Z^{5} +{\beta_{6,3}}Y^{3}
Z^{3} + {\beta_{2,1}}X^{4} Y Z  +  X^{2}{\left({\beta_{4,0}}Z^{4}  +
{\beta_{4,2}}Y^{2} Z^{2} + {\beta_{4,4}}Y^{4} \right)} $ \\\hline
$3,(0,1)$& $Z^{6} + Z^{3}{L_{3,Z}}  + {L_{6,Z}}$ \\\hline $2,(0,1)$&
$Z^{6} + Z^{4}{L_{2,Z}}  + Z^{2}{L_{4,Z}}+{L_{6,Z}} $
\\\hline
   \end{tabular}
\end{table}

\end{center}

\end{small}

%\begin{rem} \begin{itemize}
             % \item {In degree table: Types $4,(2,3)$ and $4,(1,3)$ are $K$-isomorphic to Type $4, (1,2)$ through the transformation $[X;Z;Y]$ and $[Y;X;Z]$ respectively, hence could be omitted.}
              %\item {In degree 5 table: Types $5, (1,3)$ and $5, (1,4)$ are isomorphic to Type $5, (1,2)$ through the transformations $[X;Z;Y]$ and $[Y;X;Z]$ respectively. Also Type $4, (2,3)$ is isomorphic to Type $4, (1,2)$ through the transformation $[X;\mu^{-1}Z;\mu Y]$ where $\mu^4=\alpha$. On the other hand, Type $4, (1,3)$ is of the form $X.G(X;Y;Z)$ (i.e is irreducible) therefore is a singular curve which is out of scope of this work.}
           % \item {In degree 6 table: Type $5, (1,3)$ is isomorphic to Type $5, (1,2)$ through the transformation $[X;\mu^{-1}Z;\mu Y]$ where $\mu^5=\alpha$. Also Type $6,(1,4)$ (resp. $6,(1,5)$) is isomorphic to Type $6,(1,3)$ (resp. $6,(1,2)$) through $[Y;X;Z]$ moreover type $6,(2,3)$ (resp. $6,(2,5)$) is isomorphic to $6,(1,3)$ through $[Z;Y;X]$ (resp. $[Y;Z;X]$). Finally, types $6,(3,4)$ and $6,(3,5)$ are isomorphic to $6,(1,3)$ through $[Z;X;Y]$ and $[X;Z;Y]$ and type $6,(4,5)$ is isomorphic to $6,(1,2)$ through $[X;Z;Y]$}
%            \ item In degree 7 table: Type
 %           \end{itemize}
%\end{rem}

\begin{small}
\begin{center}
\begin{table}[!th]
  \renewcommand{\arraystretch}{1.3}
  \caption{degree 7\,\,\,}\label{table:Cyclic Auto72.}
  \vspace{4mm} % hack
  \centering
\begin{tabular}{|c|c|}
  \hline
  Type: $m, (a,\,b)$ & $F(X;Y;Z)$ \\\hline\hline
  $42,(6,7)$& $X^7+ Y^7+\alpha XZ^6$ \\\hline
$36,(1,30)$& $X^{7} + Y^{6} Z+\alpha X Z^{6}$ \\\hline
  $35,(1,29)$& $X^7+Y^6Z+\alpha YZ^6$ \\\hline
  $31,(1,26)$& $X^{6} Y + Y^{6} Z+\alpha X Z^{6} $ \\\hline
  $21,(3,7)$& $X^7 + Y^7+\alpha XZ^6 + \beta_{3,0}X^4Z^3$ \\\hline
  $18,(1,12)$& $X^{7} + Y^{6} Z+\alpha X Z^{6}+ {\beta_{3,0}}X^{4} Z^{3} $ \\\hline
 $14,(2,7)$& $X^7 + Y^7+\alpha XZ^6 + \beta_{2,0}X^5Z^2 +\beta_{4,0}X^3Z^4$ \\\hline
  $12,(1,6)$& $X^{7} + Y^{6} Z+\alpha X Z^{6}+ {\beta_{2,0}}X^{5} Z^{2}+ {\beta_{4,0}}X^{3} Z^{4}$ \\\hline
  $9,(1,3)$& $X^{7} + Y^{6} Z+\alpha X Z^{6}+ {\beta_{3,0}}X^{4} Z^{3}  + {\beta_{5,3}}X^{2} Y^{3} Z^{2} $ \\\hline
  $7,(1,2)$& $X^{7} + Y^{7} + Z^{7}+{\beta_{4,1}}X^{3} Y Z^{3}  + {\beta_{5,3}}X^{2} Y^{3} Z^{2}  + {\beta_{6,5}}X Y^{5} Z $ \\\hline
 $7,(1,3)$& $X^{7} + Y^{7} + Z^{7}+{\beta_{3,1}}X^{4} Y Z^{2}  + {\beta_{5,4}}X^{2} Y^{4} Z  + {\beta_{6,2}}X Y^{2} Z^{4}  $ \\\hline
% isomorphic to $7,(1,2)$ through $[X,Z,Y]$ $(1, 4) 7$ & $X^{3} Y^{3} Z {{\beta_{4,3}}} + X^{2} Y^{2} Z^{3} {{\beta_{5,2}}} + X Y Z^{5} {{\beta_{6,1}}} + X^{7} + Y^{7} + Z^{7}$\\\hline
% isomorphic to $7,(1,3)$ through $[X,Z,Y]$$(1, 5) 7$ & $X^{4} Y^{2} Z {{\beta_{3,2}}} + X^{2} Y Z^{4} {{\beta_{5,1}}} + X Y^{4} Z^{2} {{\beta_{6,4}}} + X^{7} + Y^{7} + Z^{7}$\\\hline
% isomorphic to $7,(1,2)$ through $[Y,X,Z]$$(1, 6) 7$ & $X^{5} Y Z {{\beta_{2,1}}} + X^{3} Y^{2} Z^{2} {{\beta_{4,2}}} + X Y^{3} Z^{3} {{\beta_{6,3}}} + X^{7} + Y^{7} + Z^{7}$\\\hline
$7,(0,1)$& $Z^{7}+L_{7,Z}$ \\\hline

 \end{tabular}
\end{table}
\end{center}
\end{small}

\begin{small}
\begin{center}
\begin{tabular}{|c|c|}
  \hline
  %Type: $m, (a,b)$ & $F(X;Y;Z)$ \\\hline\hline

  $6,(1,2)$& $X^{7} + X Z^{6}+\alpha X Y^{6}+ {\beta_{3,0}}X^{4} Z^{3}+ {\beta_{4,2}}X^{3} Y^{2} Z^{2}+ {\beta_{5,4}}X^{2} Y^{4} Z+ {\beta_{7,2}}Y^{2} Z^{5}$ \\\hline
  $6,(2,3)$& $X^{7} + X Z^{6}+\alpha X Y^{6}+ {\beta_{2,0}}X^{5} Z^{2}  + {\beta_{3,3}}X^{4} Y^{3}  + {\beta_{4,0}}X^{3} Z^{4}  + {\beta_{5,3}}X^{2}Y^{3} Z^{2}+ {\beta_{7,3}}Y^{3} Z^{4}$ \\\hline
% of the form $X.G$ $\textbf{6, (1, 3)}$ & $X^{7} + X Z^{6} +\alpha X Y^{6} + {{\beta_{2,0}}} X^{5} Z^{2} + {{\beta_{6,3}}}X Y^{3} Z^{3}  +  X^{3}{\left({{\beta_{4,0}}} Z^{4} + {{\beta_{4,3}}}Y^{3} Z \right)}$\\\hline
% of the form $X.G$$\textbf{6, (1, 4)}$& $X^{7} + X Z^{6} + \alpha X Y^{6} + X^{4}{\left({{\beta_{3,0}}}Z^{3}  + {{\beta_{3,2}}}Y^{2} Z \right)}  + {\left({{\beta_{6,2}}}Y^{2} Z^{4}  + {{\beta_{6,4}}}Y^{4} Z^{2} \right)} X$\\\hline
% of the form $X.G$$\textbf{6, (1, 5)}$ &$X^{7} + X Z^{6}+\alpha X Y^{6}  + {{\beta_{2,1}}}X^{5} Y Z  + {{\beta_{4,2}}}X^{3} Y^{2} Z^{2}  + {{\beta_{6,3}}}X Y^{3} Z^{3}$\\\hline
% of the form $X.G$$\textbf{6, (2, 5)}$& $X^{7} + X Z^{6} + \alpha X Y^{6} +  X^{4}\left(\beta_{3,1} Y Z^{2} + {{\beta_{3,3}}}Y^{3} \right)  + X\left({{\beta_{6,2}}}Y^{2} Z^{4}  + {{\beta_{6,4}}}Y^{4} Z^{2} \right)$ \\\hline
% isomorphic to $6,(2,3)$ through $[X,Z,Y]$$\textbf{6, (3, 4)}$ & $ X^{7} + X Z^{6}+\alpha X Y^{6}  + {{\beta_{2,2}}}X^{5} Y^{2}  + {{\beta_{3,0}}}X^{4} Z^{3}  + {{\beta_{4,4}}}X^{3} Y^{4}  + {{\beta_{5,2}}}X^{2} Y^{2} Z^{3}  + {{\beta_{7,4}}}Y^{4} Z^{3} $ \\\hline
% of the form $X.G$$\textbf{6,(3, 5)}$ & $X^{7} + X Z^{6} + X Y^{6} \alpha + X^{5} Y^{2} {{\beta_{2,2}}} + X Y^{3} Z^{3} {{\beta_{6,3}}} +  {\left(Y Z^{3} {{\beta_{4,1}}} + Y^{4} {{\beta_{4,4}}}\right)} X^{3}$\\\hline
% isomorphic to $6,(1,2)$ through $[X,Z,Y]$$\textbf{6, (4, 5)}$ & $X^{7} + X Z^{6}+X Y^{6} \alpha + X^{4} Y^{3} {{\beta_{3,3}}} + X^{3} Y^{2} Z^{2} {{\beta_{4,2}}} + X^{2} Y Z^{4} {{\beta_{5,1}}} + Y^{5} Z^{2} {{\beta_{7,5}}} $\\\hline
$6,(0,1)$& $Z^{6}L_{1,Z} + L_{7,Z}$ \\\hline
 $5,(1,4)$& $X^7+Y^6Z+\alpha YZ^6+\beta_{2,1}X^{5}YZ+ \beta_{4,2}X^{3} Y^{2} Z^{2}+ \beta_{6,3}X Y^{3} Z^{3}+X^{2}{\left({\beta_{5,0}}Z^{5} + {\beta_{5,5}}Y^{5} \right)}  $ \\\hline
 $4,(1,2)$& $ X^{7} + Y^{6} Z+\alpha X Z^{6}  + {\beta_{2,0}}X^{5} Z^{2}  + {\beta_{3,2}}X^{4} Y^{2} Z  + {\beta_{5,2}}X^{2} Y^{2} Z^{3}  + {\beta_{6,4}}X Y^{4} Z^{2} + {\beta_{7,2}}Y^{2} Z^{5}+ $ \\
 & $+ X^{3}{\left({\beta_{4,0}}Z^{4}  + {\beta_{4,4}}Y^{4} \right)}$ \\\hline
 $3,(1,2)$& $X^{7} + X Z^{6}+\alpha X Y^{6}  + {\beta_{2,1}}X^{5} Y Z  + {\beta_{4,2}}X^{3} Y^{2} Z^{2}  + {\beta_{6,3}}X Y^{3} Z^{3}  + {\beta_{7,2}}Y^{2} Z^{5} + {\beta_{7,5}}Y^{5} Z^{2}+$ \\
 & $X^{4}\left(\beta_{3,0}Z^{3} + {\beta_{3,3}}Y^{3} \right)  + X^{2}{\left({\beta_{5,1}}Y Z^{4} + {\beta_{5,4}}Y^{4} Z \right)}$ \\\hline
 $3,(0,1)$& $Z^{6}L_{1,Z}+Z^3L_{4,Z} + L_{7,Z}$ \\\hline
 $2,(0,1)$& $Z^{6}L_{1,Z}+Z^4L_{3,Z}+Z^2L_{5,Z} + L_{7,Z}$ \\\hline

  \end{tabular}
%\end{table}
\end{center}
%
%
%\begin{center}
%\begin{tabular}{|c|c|}
%  \hline
%  \end{tabular}
%\end{table}
%
%\end{center}
%
\end{small}
\begin{small}
\begin{center}
\begin{table}[!th]
  \renewcommand{\arraystretch}{1.3}
  \caption{degree 8\,\,\,}\label{table:Cyclic Auto82.}
  \vspace{4mm} % hack
  \centering
\begin{tabular}{|c|c|}
  \hline
  Type: $m, (a,\,b)$ & $F(X;Y;Z)$ \\\hline\hline
  $56,(7,8)$& $X^{8} + Y^{8}+\alpha X Z^{7}  $ \\\hline
  $49,(1,42)$& $ X^{8} + Y^{7} Z+\alpha X Z^{7}  $ \\\hline
$48,(1,41)$& $X^{8} + Y^{7} Z+\alpha Y Z^{7} $ \\\hline
$43,(1,37)$&$ X^{7} Y + Y^{7} Z+\alpha X Z^{7} $ \\\hline
$28,(7,8)$& $X^{8}+Y^{8}+\alpha X Z^{7}  + {\beta_{4,4}}X^{4} Y^{4} $ \\\hline
$24,(1,17)$& $X^{8} + Y^{7} Z+\alpha Y Z^{7}+{\beta_{8,4}}Y^{4}
Z^{4} $ \\\hline
$16,(1,9)$& $X^{8} + Y^{7} Z+\alpha YZ^{7}+{\beta_{8,5}}Y^{5} Z^{3}  + {\beta_{8,3}}Y^{3} Z^{5} $
\\\hline

 % \end{tabular}
%\end{table}
%\end{center}
%
%
%\begin{center}
%\begin{tabular}{|c|c|}
%  \hline

 % \end{tabular}
%\end{table}
%
%\end{center}
%
%\end{small}
%
%
%\begin{small}
%\begin{center}
%\begin{tabular}{|c|c|}
%  \hline
$14,(7,8)$& $X^{8} + Y^{8}+\alpha X Z^{7}  +
{\beta_{2,2}}X^{6} Y^{2}  + {\beta_{4,4}}X^{4} Y^{4}  +
{\beta_{6,6}}X^{2} Y^{6} $ \\\hline

$12,(1,5)$& $X^{8} + Y^{7}
Z+\alpha Y Z^{7} +{\beta_{8,4}}Y^{4} Z^{4}  + {\beta_{4,2}}X^{4}
Y^{2} Z^{2}$\\\hline
$8,(1,2)$& $X^{8} + Y^{8} +
Z^{8}+{\beta_{4,0}} X^{4} Z^{4} + {\beta_{5,2}}X^{3} Y^{2}
Z^{3}+{\beta_{6,4}}X^{2} Y^{4} Z^{2}+{\beta_{7,6}}X Y^{6} Z $
\\\hline
$8,(1,3)$& $X^{8} + Y^{8} + Z^{8} +{\beta_{4,2}}X^{4} Y^{2}
Z^{2}  + {\beta_{8,4}}Y^{4} Z^{4}  +  X^{2}{\left({\beta_{6,1}}Y
Z^{5}  + {\beta_{6,5}}Y^{5} Z \right)} $ \\\hline
$8,(1,4)$& $X^{8}
+ Y^{8} + Z^{8}+{\beta_{2,0}}X^{6} Z^{2}  + {\beta_{4,0}}X^{4} Z^{4}
+ {\beta_{5,4}}X^{3} Y^{4} Z  + {\beta_{6,0}}X^{2} Z^{6} +
{\beta_{7,4}}X Y^{4} Z^{3} $ \\\hline
%isomorphic to $8(1,4)$ through $[Y,X,Z]$$(1, 5) 8$ & $Y^{2} Z^{6} {{\beta_{8,2}}} + Y^{4} Z^{4} {{\beta_{8,4}}} + Y^{6} Z^{2} {{\beta_{8,6}}} + X^{8} + Y^{8} + Z^{8} + {\left(Y Z^{3} {{\beta_{4,1}}} + Y^{3} Z {{\beta_{4,3}}}\right)} X^{4}$\\\hline
%isomorphic to $8(1,3)$ through $[Y,X,Z]$$(1, 6) 8$ & $X^{5} Y^{2} Z {{\beta_{3,2}}} + X^{4} Z^{4} {{\beta_{4,0}}} + X^{2} Y^{4} Z^{2} {{\beta_{6,4}}} + X Y^{2} Z^{5} {{\beta_{7,2}}} + X^{8} + Y^{8} + Z^{8}$\\\hline
%isomorphic to $8(1,2)$ through $[Y,X,Z]$$(1, 7) 8$ & $X^{6} Y Z {{\beta_{2,1}}} + X^{4} Y^{2} Z^{2} {{\beta_{4,2}}} + X^{2} Y^{3} Z^{3} {{\beta_{6,3}}} + Y^{4} Z^{4} {{\beta_{8,4}}} + X^{8} + Y^{8} + Z^{8}$\\\hline
%isomorphic to $8(1,3)$ through $[Z,Y,X]$$(2, 3) 8$ &$X^{5} Y Z^{2} {{\beta_{3,1}}} + X^{4} Y^{4} {{\beta_{4,4}}} + X^{2} Y^{2} Z^{4} {{\beta_{6,2}}} + X Y^{5} Z^{2} {{\beta_{7,5}}} + X^{8} + Y^{8} + Z^{8}$\\\hline
%isomorphic to $8(1,2)$ through $[X,Z,Y]$$(2, 5) 8$ & $X^{4} Y^{4} {{\beta_{4,4}}} + X^{3} Y^{3} Z^{2} {{\beta_{5,3}}} + X^{2} Y^{2} Z^{4} {{\beta_{6,2}}} + X Y Z^{6} {{\beta_{7,1}}} + X^{8} + Y^{8} + Z^{8}$\\\hline
%isomorphic to $8(1,4)$ through $[X,Z,Y]$$(4, 5) 8$ & $X^{6} Y^{2} {{\beta_{2,2}}} + X^{4} Y^{4} {{\beta_{4,4}}} + X^{3} Y Z^{4} {{\beta_{5,1}}} + X^{2} Y^{6} {{\beta_{6,6}}} + X Y^{3} Z^{4} {{\beta_{7,3}}} + X^{8} + Y^{8} + Z^{8}$\\\hline
$8,(0,1)$& $Z^{8} + {L_{8,Z}}$
\\\hline
$7,(1,2)$& $X^{8} + X Z^{7}+\alpha X Y^{7}  +
{\beta_{4,1}}X^{4} Y Z^{3}  + {\beta_{5,3}}X^{3} Y^{3} Z^{2}  +
{\beta_{6,5}}X^{2} Y^{5} Z  + {\beta_{8,2}} Y^{2} Z^{6}$ \\\hline
$7,(1,3)$& $X^{8} + X Z^{7}+\alpha X Y^{7}  + {\beta_{3,1}}X^{5} Y
Z^{2}  + {\beta_{5,4}}X^{3} Y^{4} Z  + {\beta_{6,2}}X^{2} Y^{2}
Z^{4}  + {\beta_{8,5}}Y^{5} Z^{3}$ \\\hline
% isomorphic to $7,(1,2)$ through $[x,z,y]$$(1, 4) 7$ & $X Y^{7} \alpha + X^{4} Y^{3} Z {{\beta_{4,3}}} + X^{3} Y^{2} Z^{3} {{\beta_{5,2}}} + X^{2} Y Z^{5} {{\beta_{6,1}}} + Y^{6} Z^{2} {{\beta_{8,6}}} + X^{8} + X Z^{7}$\\\hline
% isomorphic to $7,(1,3)$ through $[x,z,y]$$(1, 5) 7$ & $X Y^{7} \alpha + X^{5} Y^{2} Z {{\beta_{3,2}}} + X^{3} Y Z^{4} {{\beta_{5,1}}} + X^{2} Y^{4} Z^{2} {{\beta_{6,4}}} + Y^{3} Z^{5} {{\beta_{8,3}}} + X^{8} + X Z^{7}$\\\hline
$7, (1, 6)$ & $X^{8} + X Z^{7}+\alpha X Y^{7} + {{\beta_{2,1}}} X^{6} Y Z + {{\beta_{4,2}}}X^{4} Y^{2} Z^{2} + {{\beta_{6,3}}}X^{2} Y^{3} Z^{3}  + {{\beta_{8,4}}}Y^{4} Z^{4}$\\\hline
$7,(0,1)$&$Z^{7}{L_{1,Z}}  + {L_{8,Z}}$ \\\hline
$6,(1,5)$& $X^{8} + Y^{7} Z +\alpha Y Z^{7}  + {\beta_{2,1}} X^{6} Y Z  + {\beta_{4,2}}X^{4} Y^{2} Z^{2} +{\beta_{8,4}}Y^{4} Z^{4} $ \\
& $+X^{2}{\left({\beta_{6,0}}Z^{6}  + {\beta_{6,3}}Y^{3} Z^{3}  +
{\beta_{6,6}}Y^{6} \right)}  $ \\\hline $4,(0,1)$& $Z^{8} +
Z^{4}{L_{4,Z}}  + {L_{8,Z}}$ \\\hline
$3,(1,2)$& $X^{8} + Y^{7} Z +\alpha Y Z^{7}  +{\beta_{8,4}}Y^{4} Z^{4}  +  {\beta_{2,1}}X^{6} Y Z  + {\beta_{4,2}}X^{4} Y^{2} Z^{2}  +  X^{5}{\left({\beta_{3,0}}Z^{3}  + {\beta_{3,3}}Y^{3} \right)} + $ \\
& $+ X^{3}{\left({\beta_{5,1}}Y Z^{4}  + {\beta_{5,4}}Y^{4} Z
\right)}  + X^{2}{\left({\beta_{6,0}}Z^{6}  + {\beta_{6,3}}Y^{3}
Z^{3}  + {\beta_{6,6}}Y^{6} \right)}  + X{\left({\beta_{7,2}}Y^{2}
Z^{5}  + {\beta_{7,5}}Y^{5} Z^{2} \right)} $ \\\hline $2,(0,1)$&
$Z^{8} + Z^{6}{L_{2,Z}}  + Z^{4}{L_{4,Z}}  + Z^{2}{L_{6,Z}}  +
{L_{8,Z}}$ \\\hline

 % \end{tabular}
%\end{table}
%\end{center}
%
%
%\begin{center}
%\begin{tabular}{|c|c|}
%  \hline

  \end{tabular}
\end{table}
\end{center}

\end{small}

\newpage

\begin{small}
\begin{center}
\begin{table}[!th]
  \renewcommand{\arraystretch}{1.3}
  \caption{degree 9\,\,\,}\label{table:Cyclic Auto92.}
  \vspace{4mm} % hack
  \centering
\begin{tabular}{|c|c|}
  \hline
  Type: $m, (a,\,b)$ & $F(X;Y;Z)$ \\\hline\hline
$72,(8,9)$& $X^{9} + Y^{9}+\alpha X Z^{8}$ \\\hline $64,(1,56)$&
$X^{9} + Y^{8} Z+\alpha X Z^{8}$ \\\hline $63,(1,55)$& $X^{9} +
Y^{8} Z+\alpha Y Z^{8}$ \\\hline $57,(1,50)$& $X^{8} Y + Y^{8}
Z+\alpha X Z^{8} $ \\\hline $36,(4,9)$& $X^{9} + Y^{9}+\alpha X
Z^{8}+ {\beta_{4,0}}X^{5} Z^{4} $ \\\hline $32,(1,24)$& $X^{9} +
Y^{8} Z+\alpha X Z^{8}+ {\beta_{4,0}}X^{5} Z^{4} $ \\\hline
$24,(8,9)$& $X^{9} + Y^{9}+\alpha X Z^{8}+ {\beta_{3,3}}X^{6} Y^{3}
+ {\beta_{6,6}}X^{3} Y^{6} $ \\\hline $21,(1,13)$& $X^{9} + Y^{8}
Z+\alpha Y Z^{8}  + {\beta_{6,3}}X^{3} Y^{3} Z^{3} $ \\\hline
$18,(2,9)$& $X^{9} + Y^{9}+\alpha X Z^{8}+ {\beta_{2,0}}X^{7} Z^{2}
+ {\beta_{4,0}}X^{5} Z^{4}  + {\beta_{6,0}}X^{3} Z^{6} $ \\\hline
$16,(1,8)$& $X^{9} + Y^{8} Z+\alpha X Z^{8}+ {\beta_{2,0}}X^{7}
Z^{2}  + {\beta_{4,0}}X^{5} Z^{4}  + {\beta_{6,0}}X^{3} Z^{6}  $
\\\hline
$12,(4,9)$& $X^{9} + Y^{9}+\alpha X Z^{8}  +
{\beta_{3,3}}X^{6} Y^{3}  + {\beta_{4,0}}X^{5} Z^{4}  +
{\beta_{6,6}}X^{3} Y^{6}  + {\beta_{7,3}}X^{2} Y^{3} Z^{4} $
\\\hline
$9,(1,2)$& $X^{9} + Y^{9} + Z^{9}+{\beta_{5,1}}X^{4} Y
Z^{4}  + {\beta_{6,3}}X^{3} Y^{3} Z^{3}  + {\beta_{7,5}}X^{2} Y^{5}
Z^{2}  + {\beta_{8,7}}X Y^{7} Z $ \\\hline
$9,(1,3)$& $X^{9} + Y^{9}
+ Z^{9}+{\beta_{3,0}}X^{6} Z^{3}  + {\beta_{5,3}}X^{4} Y^{3} Z^{2}
+ {\beta_{6,0}}X^{3} Z^{6}  + {\beta_{7,6}}X^{2} Y^{6} Z +
{\beta_{8,3}}X Y^{3} Z^{5}  $ \\\hline
$9,(0,1)$& $Z^{9} +{L_{9,Z}}$ \\\hline
$8,(1,2)$& $X^{9} + X Z^{8}+\alpha X Y^{8}  +
{\beta_{4,0}}X^{5} Z^{4}  + {\beta_{5,2}}X^{4} Y^{2} Z^{3}  +
{\beta_{6,4}}X^{3} Y^{4} Z^{2}  + {\beta_{7,6}}X^{2} Y^{6} Z  +
{\beta_{9,2}}Y^{2} Z^{7}$ \\\hline
$8,(1,4)$& $X^{9} + XZ^{8}+\alpha X Y^{8} + {\beta_{2,0}}X^{7} Z^{2}  +
{\beta_{4,0}}X^{5} Z^{4}  + {\beta_{5,4}}X^{4} Y^{4} Z  +
{\beta_{6,0}}X^{3} Z^{6}  + {\beta_{7,4}}X^{2} Y^{4} Z^{3}  +
{\beta_{9,4}}Y^{4} Z^{5} $ \\\hline
$8,(1,6)$& $X^{9} + XZ^{8}+\alpha X Y^{8} +{\beta_{3,2}}X^{6} Y^{2} Z  +
{\beta_{4,0}}X^{5} Z^{4}  +  {\beta_{6,4}}X^{3} Y^{4} Z^{2} +
{\beta_{7,2}}X^{2} Y^{2} Z^{5} +  {\beta_{9,6}}Y^{6} Z^{3} $
\\\hline
%of the form $X.G$$(1, 3) 8$ & $X Y^{8} \alpha + X^{5} Y^{2} Z^{2} {{\beta_{4,2}}} + X Y^{4} Z^{4} {{\beta_{8,4}}} + X^{9} + X Z^{8} + {\left(Y Z^{5} {{\beta_{6,1}}} + Y^{5} Z {{\beta_{6,5}}}\right)} X^{3}$\\\hline
%of the form $X.G$$(1, 5) 8$ & $X Y^{8} \alpha + X^{9} + X Z^{8} + {\left(Y Z^{3} {{\beta_{4,1}}} + Y^{3} Z {{\beta_{4,3}}}\right)} X^{5} + {\left(Y^{2} Z^{6} {{\beta_{8,2}}} + Y^{4} Z^{4} {{\beta_{8,4}}} + Y^{6} Z^{2} {{\beta_{8,6}}}\right)} X$\\\hline
%of the form $X.G$$(1, 7) 8$ & $ X Y^{8} \alpha + X^{7} Y Z {{\beta_{2,1}}} + X^{5} Y^{2} Z^{2} {{\beta_{4,2}}} + X^{3} Y^{3} Z^{3} {{\beta_{6,3}}} + X Y^{4} Z^{4} {{\beta_{8,4}}} + X^{9} + X Z^{8}$\\\hline
% isomorphic to $8, (1,6)$ through $[X,Z,Y]$ $(2, 3) 8$ & $X Y^{8} \alpha + X^{6} Y Z^{2} {{\beta_{3,1}}} + X^{5} Y^{4} {{\beta_{4,4}}} + X^{3} Y^{2} Z^{4} {{\beta_{6,2}}} + X^{2} Y^{5} Z^{2} {{\beta_{7,5}}} + Y^{3} Z^{6} {{\beta_{9,3}}} + X^{9} + X Z^{8}$\\\hline
% isomorphic to $8, (1,2)$ through $[X,Z,Y]$$(2, 5) 8$ & $X Y^{8} \alpha + X^{5} Y^{4} {{\beta_{4,4}}} + X^{4} Y^{3} Z^{2} {{\beta_{5,3}}} + X^{3} Y^{2} Z^{4} {{\beta_{6,2}}} + X^{2} Y Z^{6} {{\beta_{7,1}}} + Y^{7} Z^{2} {{\beta_{9,7}}} + X^{9} + X Z^{8}$\\\hline
% isomorphic to $8, (1,2)$ through $[X,Z,Y]$$(4, 5) 8$ & $X Y^{8} \alpha + X^{7} Y^{2} {{\beta_{2,2}}} + X^{5} Y^{4} {{\beta_{4,4}}} + X^{4} Y Z^{4} {{\beta_{5,1}}} + X^{3} Y^{6} {{\beta_{6,6}}} + X^{2} Y^{3} Z^{4} {{\beta_{7,3}}} + Y^{5} Z^{4} {{\beta_{9,5}}} + X^{9} + X Z^{8}$\\\hline
$8,(0,1)$& $Z^{8}{L_{1,Z}}+{L_{9,Z}}$ \\\hline

 % \end{tabular}
%\end{table}
%\end{center}
%
%
%\begin{center}
%\begin{tabular}{|c|c|}
%  \hline
$7,(1,6)$& $X^{9} + Y^{8} Z +\alpha Y Z^{8}  + {\beta_{2,1}}X^{7} Y Z  + {\beta_{4,2}}X^{5} Y^{2} Z^{2}  + {\beta_{6,3}}X^{3} Y^{3} Z^{3}  + $ \\
& $+{\beta_{8,4}}X Y^{4} Z^{4}+ X^{2}\left({\beta_{7,0}}Z^{7}  + {\beta_{7,7}}Y^{7} \right)$ \\\hline
$6,(2,3)$& $X^{9} + Y^{9} +\alpha X Z^{8}  + {\beta_{2,0}}X^{7} Z^{2}  + {\beta_{3,3}}X^{6} Y^{3}  + {\beta_{4,0}}X^{5} Z^{4}  + {\beta_{5,3}}X^{4} Y^{3} Z^{2}+ $ \\
& $+{\beta_{7,3}}X^{2} Y^{3} Z^{4}  + {\beta_{7,3}}Y^{3} Z^{6}  +  {\beta_{8,6}}Y^{6} Z^{3}  + X^{3}{\left({\beta_{6,0}}Z^{6}  + {\beta_{6,6}}Y^{6} \right)} $ \\\hline

%
%\begin{small}
%\begin{center}
%\begin{tabular}{|c|c|}
%  \hline
%  % after \\: \hline or \cline{col1-col2} \cline{col3-col4} ...

$4,(1,2)$& $X^{9} + X Z^{8} +\alpha X Y^{8}  + {\beta_{2,0}}X^{7} Z^{2}  + {\beta_{3,2}}X^{6} Y^{2} Z  + {\beta_{5,2}}X^{4} Y^{2} Z^{3}  + {\beta_{8,4}}X Y^{4} Z^{4}  + $ \\
& ${\beta_{9,2}}Y^{2} Z^{7}  + {\beta_{9,6}}Y^{6} Z^{3} +
X^{5}{\left({\beta_{4,0}}Z^{4}  + {\beta_{4,4}}Y^{4} \right)}  +
X^{3}{\left({\beta_{6,0}}Z^{6}  + {\beta_{6,4}}Y^{4} Z^{2} \right)}
+ X^{2}{\left({\beta_{7,2}}Y^{2} Z^{5}  + {\beta_{7,6}}Y^{6} Z
\right)} $ \\\hline
% of the form $X.G$ $(1, 3) 4$ & $X Y^{8} \alpha + X^{7} Y Z {{\beta_{2,1}}} + X^{9} + X Z^{8} + {\left(Z^{4} {{\beta_{4,0}}} + Y^{2} Z^{2} {{\beta_{4,2}}} + Y^{4} {{\beta_{4,4}}}\right)} X^{5} + {\left(Y Z^{5} {{\beta_{6,1}}} + Y^{3} Z^{3} {{\beta_{6,3}}} + Y^{5} Z {{\beta_{6,5}}}\right)} X^{3} + {\left(Y^{2} Z^{6} {{\beta_{8,2}}} + Y^{4} Z^{4} {{\beta_{8,4}}} + Y^{6} Z^{2} {{\beta_{8,6}}}\right)} X$\\\hline
%isomorphic to $4,(1,2)$ through $[X,Z,Y]$$(2, 3) 4$ & $X Y^{8} \alpha + X^{7} Y^{2} {{\beta_{2,2}}} + X^{6} Y Z^{2} {{\beta_{3,1}}} + X^{4} Y^{3} Z^{2} {{\beta_{5,3}}} + X Y^{4} Z^{4} {{\beta_{8,4}}}+Y^{3} Z^{6} {{\beta_{9,3}}} + Y^{7} Z^{2} {{\beta_{9,7}}} +$\\
%& $+ X^{9} + X Z^{8} + {\left(Z^{4} {{\beta_{4,0}}} + Y^{4} {{\beta_{4,4}}}\right)} X^{5} + {\left(Y^{2} Z^{4} {{\beta_{6,2}}} + Y^{6} {{\beta_{6,6}}}\right)} X^{3} + {\left(Y Z^{6} {{\beta_{7,1}}} + Y^{5} Z^{2} {{\beta_{7,5}}}\right)} X^{2}$\\\hline
$4,(0,1)$& $Z^{8}{L_{1,Z}}+ Z^{4}{L_{5,Z}}+
{L_{9,Z}}$ \\\hline $3,(0,1)$& $Z^{9} + Z^{6}{L_{3,Z}}  +
Z^{3}{L_{6,Z}}  + {L_{9,Z}}$ \\\hline $2,(0,1)$& $Z^{8}{L_{1,Z}} +
Z^{6}{L_{3,Z}}  + Z^{4}{L_{5,Z}}  + Z^{2}{L_{7,Z}}  + {L_{9,Z}}$
\\\hline

 % \end{tabular}
%\end{table}
%\end{center}
%
%
%\begin{center}
%\begin{tabular}{|c|c|}
%  \hline
 \end{tabular}
\end{table}

\end{center}

\end{small}

\end{document}